\documentclass[letterpaper,12pt]{article}

\usepackage{stix}  

\usepackage[titletoc,toc,title]{appendix}	

\usepackage{fullpage}		  

\usepackage{amsfonts}

\usepackage{amssymb}
\usepackage{amsmath}
\usepackage{amsthm}
\usepackage[usenames]{color}
\usepackage{amscd}
\usepackage{verbatim}

\usepackage{graphicx}
\usepackage{caption}		
\usepackage{subcaption} 

\usepackage{tikz}						
\usetikzlibrary{calc}
\usetikzlibrary{shapes}
\usetikzlibrary{plotmarks}	

\usepackage{url}

\usepackage{float}		
\usepackage[colorlinks=true, citecolor=blue, linktocpage]{hyperref}

\usepackage{enumerate}	

\usepackage[makeroom]{cancel}	

\usepackage{mathtools}	
\usepackage{stmaryrd}		

\usepackage{soul}				

\usepackage{listings,xcolor}										
\usepackage{blkarray}		

\usepackage{mdframed}		

\definecolor{shadecolor}{gray}{0.90}				
\def\boitegrise#1#2{\begin{centerline}{\fcolorbox{black}{shadecolor}{~
    \begin{minipage}[t]{#2}{\vphantom{~}#1\vphantom{$A_{\displaystyle{A_A}}$}}
            \end{minipage}~}}\end{centerline}\medskip}

\usepackage[autolanguage]{numprint}  

\newtheorem{theorem}{Theorem}[section] 

\newtheorem{proposition}[theorem]{Proposition}

\theoremstyle{definition}
\newtheorem{example}[theorem]{Example}

\theoremstyle{definition}
\newtheorem{definition}[theorem]{Definition}

\theoremstyle{definition}
\newtheorem{observation}[theorem]{Observation}

\theoremstyle{definition}
\newtheorem{remark}[theorem]{Remark}

\theoremstyle{definition}

\theoremstyle{definition}
\newtheorem{convention}[theorem]{Convention}
\newtheorem{question}{Question}

\newcommand{\nocontentsline}[3]{}
\newcommand{\tocless}[2]{\bgroup\let\addcontentsline=\nocontentsline#1{#2}\egroup}

\DeclareMathOperator{\Mat}{Mat}		
\DeclareMathOperator{\QP}{\mathit{Q{\mkern-1.7mu}P}_{\!\!\textit{M}}}

\DeclareMathOperator{\ggcd}{gcd}

\newcommand{\Hforest}[1]{H_{(x, y)}^{#1}}       

\newcommand{\RotateRight}{\mbox{\scalebox{1.2}{$\circlearrowright$}}}
\newcommand{\RotateLeft}{\mbox{\scalebox{1.2}{$\circlearrowleft$}}}

\newcommand{\optGcd}{\mbox{\footnotesize opt-}\mathrm{Gcd}_M}

\newcommand{\tikzcircle}[2][red,fill=red]{\tikz[baseline=-0.5ex]\draw[#1,radius=#2] (0,0) circle ;}

\numberwithin{figure}{section}
\numberwithin{theorem}{section}

\input xy	
\xyoption{all}

\title{\textcolor{blue}{\textbf{New methods to find patches of invisible integer lattice points}}}

\author{
	Austin Goodrich\thanks{Goodrich was a student at University of Wisconsin-Eau Claire from 2012-2014, completed his bachelors in sociology and statistics at University of Wisconsin-La Crosse in 2018, and received his masters in statistics at University of Wisconsin-La Crosse in 2020. Email: \href{mailto:awgoodie@gmail.com}{\nolinkurl{awgoodie@gmail.com}}}
  \and
  aBa Mbirika\thanks{Mbirika is currently an Associate Professor of Mathematics at University of Wisconsin-Eau Claire.\newline Email: \href{mailto:mbirika@uwec.edu}{\nolinkurl{mbirika@uwec.edu}}}
  \and
	Jasmine Nielsen\thanks{Nielsen completed her bachelors in mathematics at the University of Wisconsin-Eau Claire in 2016 and plans to attend graduate school in data sciences. Email: \href{mailto:jasminemlnielsen@gmail.com}{\nolinkurl{jasminemlnielsen@gmail.com}}}
}

\date{\today}

\begin{document}

\maketitle

\begin{center}
\textit{\normalsize{\textcolor{red}{``Mathematics is the queen of the sciences and number theory is the queen of mathematics.'' \newline -- Carl Friedrich Gauss (1777-1855)}}}
\end{center}

\vfill

\begin{abstract}
It is a surprising fact that the proportion of integer lattice points visible from the origin is exactly $\frac{6}{\pi^2}$, or approximately 60 percent.  Hence, approximately 40 percent of the integer lattice is hidden from the origin.  Since 1971, many have studied a variety of problems involving lattice point visibility, in particular, searching for patterns in that 40 percent of the lattice comprised of invisible points.  One such pattern is a square patch, an $n \times n$ grid of $n^2$ invisible points, which we call a hidden forest.  It is known that there exist arbitrarily large hidden forests in the integer lattice.  However, the methods up to now involve the Chinese Remainder Theorem (CRT) on the rows and columns of matrices with prime number entries, and they have only been able to locate hidden forests very far from the origin.  For example, using this method the closest known $4 \times 4$ hidden forest is over 3 quintillion, or $3 \times 10^{18}$, units away from the origin.  We introduce the concept of quasiprime matrices and utilize a variety of computational and theoretical techniques to find some of the closest known hidden forests to this date.  Using these new techniques, we find a $4 \times 4$ hidden forest that is merely 184 million units away from the origin.  We conjecture that every hidden forest can be found via the CRT-algorithm on a quasiprime matrix.
\end{abstract}

\vfill

\newpage
\tableofcontents




\section{Introduction}

Imagine the plane $\mathbb{R}^2$ as a forest in which each non-origin lattice point in $\mathbb{Z}^2$ is a tree and each tree is infinitely thin yet also opaque.  In this scenario, we say that a tree is hidden if some other tree lies in your line of sight from the origin.

\begin{figure}
\begin{center}
\hspace{-.75in}
\begin{subfigure}[h]{0.45\textwidth}
\begin{tikzpicture}[scale=.8]
\draw [help lines] (0,0) grid (6,6);
\draw [<->, ultra thick, blue] (0,6.5) -- (0,0) -- (6.5,0);
\foreach \x in {1,...,6}
		\node [below] at (\x,0) {\x};
\foreach \x in {1,...,6}
		\node [left] at (0,\x) {\x};
\node at (0,0) {\includegraphics[width=8ex]{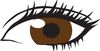}};
\node [above] at (1,3) {\includegraphics[width=2.5ex]{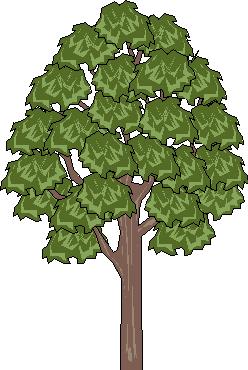}};
\node [above] at (2,1) {\includegraphics[width=2.5ex]{tree}};
\node [above] at (3,4) {\includegraphics[width=2.5ex]{tree}};
\node [above] at (5,1) {\includegraphics[width=2.5ex]{tree}};
\draw[dashed,very thick,red] (0,0) -- (2,6);
\draw[dashed,very thick,red] (0,0) -- (4.5,6);
\draw[dashed,very thick,red] (0,0) -- (6,3);
\draw[dashed,very thick,red] (0,0) -- (6,1.2);
\node [above] at (2,6) {\small(2,6)};
\node [above right] at (3,4) {\small(3,4)};
\node [right] at (6,3) {\small(6,3)};
\node [above right] at (5,1) {\small(5,1)};
\node [above right] at (2,1) {\small(2,1)};
\node [above right] at (1,3) {\small(1,3)};
\node [above right] at (4,2) {\small(4,2)};
\draw plot[only marks, mark=*, mark options={fill=white}, mark size=.6ex] coordinates {(2,6) (4,2) (6,3)};
\draw plot[only marks, mark=*, mark options={fill=black}, mark size=.6ex] coordinates {(1,3) (2,1) ((3,4) (5,1)};
\end{tikzpicture}
\caption{Four visible trees}
\label{fig:four_visible_trees}
\end{subfigure}
\begin{subfigure}[h]{0.45\textwidth}
\begin{tikzpicture}[scale=.32]
\draw [<->, ultra thick, blue] (22,0) -- (0,0) -- (0,16);
\node at (0,0) {\includegraphics[width=8ex]{eye}};
\draw[dashed,red] (0,0) -- (20,14);
\draw[dashed,red] (0,0) -- (21,14);
\draw[dashed,red] (0,0) -- (20,15);
\draw[dashed,red] (0,0) -- (21,15);
\draw[dotted] (20,14) rectangle (21,15);
\node [below right] at (10,7) {\footnotesize(10,7)};
\node [above] at (4,3) {\footnotesize(4,3)};
\node [below] at (3,2) {\footnotesize(3,2)};
\node [above] at (7,5) {\footnotesize(7,5)};
\node [below left] at (20,14) {\footnotesize(20,14)};
\node [below right] at (21,14) {\footnotesize(21,14)};
\node [above left] at (20,15) {\footnotesize(20,15)};
\node [above right] at (21,15) {\footnotesize(21,15)};
\draw plot[only marks, mark=*, mark options={fill=white}, mark size=.8ex] coordinates {(20,14) (20,15) (21,14) (21,15)};
\draw plot[only marks, mark=*, mark options={fill=black}, mark size=.8ex] coordinates {(10,7) (4,3) (3,2) (7,5)};
\end{tikzpicture}
\caption{A $2 \times 2$ hidden forest}
\label{fig:2by2_hidden_forest}
\end{subfigure}
\end{center}
\vspace{-.15in}
\caption{}
\end{figure}

Consider the four lines of sight denoted by the dashed line segments emanating from the origin in Figure~\ref{fig:four_visible_trees}.  In these four lines of sight in the first quadrant, exactly four trees are visible---one per each line of sight.  These visible trees are located at the black bullet points.  Obscured by them are three other trees at the white bullet points, which are not visible from the origin.  The tree at $(2,6)$ is obscured by the visible tree at $(1,3)$, while the tree at $(6,3)$ is obscured by the tree at $(4,2)$, which in turn is obscured by the visible tree at $(2,1)$.  The question of the visibility (or invisibility) of a lattice point from the origin can be recast in a number-theoretic setting, where it turns out that the only visible points are the points $(x,y)$ such that $\gcd(x,y) = 1$.  A proof for this visibility criterion is given in Proposition~\ref{prop:visibility_of_a_point}.

It is well known that approximately $60\%$ of the integer lattice is visible from the origin (see Proposition~\ref{prop:fraction_of_visible_points}).  So a natural question to ask about the approximately $40\%$ of the integer lattice which is hidden from view is the following:
\begin{center}
\textit{Are there arbitrarily large square patches of invisible lattice points?}
\end{center}
The answer to this question is yes, and in this paper we focus on invisible $n \times n$ square patches which we call \textit{hidden forests}.  An example of a $2 \times 2$ hidden forest is given in Figure~\ref{fig:2by2_hidden_forest}.  In this figure we note the specific four visible trees that obscure this hidden forest.

Lattice point visibility is a well-studied subject that arises in a variety of areas such as number theory, integer optimization, and even theoretical physics (see Ch.10.4 of~\cite{Brass} for a brief survey).  In 1971, Herzog and Stewart studied patterns of both visible and invisible lattice points; one such invisible pattern they explored is the one we call a hidden forest~\cite{Herzog}.  In 1990, Schumer also examined hidden forests~\cite{Schumer}.  He used the Chinese Remainder Theorem (in a form similar to our Theorem~\ref{thm:2by2_case}) and gave an example of a $3 \times 3$ hidden forest very far from the origin and questioned whether a closer one could be found.  He then noted that finding a $4 \times 4$ forest would require solving systems of linear congruence equations modulo the product of the first 16 primes, the so-called 16th \textit{primorial} which is approximately 32 quintillion, and declared, ``Such a project is beyond the courage of this author!''  In this paper we not only undertake this task of finding closer $4 \times 4$ hidden forests, but also introduce a variety of theoretical and computational techniques to aid in finding the closest known $n \times n$ hidden forests for $n \geq 4$, a task which has not yet been done to this date.  The paper is broken down as follows:
\begin{itemize}
\item Section~\ref{sec:brief_overview}: We give a brief overview of lattice point visibility and provide a detailed proof of the well-known result that the probability of two randomly selected integers being relatively prime is $\frac{6}{\pi^2}$.
\item Section~\ref{sec:the_CRT_algorithm}: We give the known method of finding hidden forests in Subsection~\ref{subsec:known_method}.  Given $n \in \mathbb{N}$ and a prime matrix $P_n$, there exists an $n \times n$ hidden forest $\Hforest{n}$ in the first quadrant with bottom-left corner $(x,y)$ that is found by applying the Chinese Remainder Theorem to the rows and columns of $P_n$.  We denote this process by the term \texttt{CRT}-algorithm. In Subsection~\ref{subsec:gcd-grid_and_matrix}, we elaborate on the relationship between a prime matrix $P_n$ and its hidden forest $\Hforest{n}$ with the introduction of an object called a \textit{gcd-matrix}.  In Subsection~\ref{subsec:applications_of_known_method}, we apply this method to find hidden forests $\Hforest{n}$ for $n=2,3,4$.
\item Section~\ref{sec:quasiprime}: We introduce the concept of a \textit{quasiprime matrix} $\QP$ and the \texttt{QP}-algorithm in Subsection~\ref{subsec:quasi} and define the method of \textit{strings of strongly composite integers} in Subsection~\ref{subsec:strings}.  In Subsection~\ref{subsec:min_prime_factors}, we explore the notion of an \textit{optimal gcd-matrix} by considering the minimal number of prime factors required in a quasiprime matrix to produce an $n \times n$ hidden forest.  With these three tools and some computational programming techniques, we then use the \texttt{CRT}-algorithm on quasiprime matrices to find $n \times n$ hidden forests. The resulting hidden forests turn out to be much closer to the origin than those found by the traditional method (given in Section~\ref{sec:the_CRT_algorithm}).

\item Section~\ref{sec:the_great_unknown}: We combine the techniques detailed in the previous section to find the closest known (to this date) $5 \times 5$ hidden forest.
\item Section~\ref{sec:open_problems}: We give a selection of open problems.  We also briefly review some recent research between second author Mbirika and collaborators Goins, Harris, and Kubik, generalizing this classic setting of straight lines of sights to curved lines of sights~\cite{GHKM2017}.
\end{itemize}

\section*{Acknowledgments}
We thank Stephan Garcia of Pomona College who introduced co-author aBa Mbirika to the concept of these hidden forests at the AIM-NSF research workshop, REUF4, at ICERM in June 2012.  We also thank the Office of Research and Sponsored Programs at the University of Wisconsin-Eau Claire (UWEC) who funded this project from Fall 2013 through Summer 2014.  We gratefully thank Zane Toman who helped us with his Java coding guidance.  Lastly, we appreciate the use of the computing resources of the Blugold Supercomputing Cluster of UWEC. Without access to its unending hard work and processing power, the immense calculations that we needed probably would not have been possible to complete within our lifetime.


\section{Density of visible lattice points in \texorpdfstring{$\mathbb{Z}^2$}{Z x Z}}\label{sec:brief_overview}

As mentioned in the introduction, a criterion for the visibility of an integer lattice point can be recast in the number-theoretic setting as the following proposition shows.

\begin{proposition}\label{prop:visibility_of_a_point}
Let $(x,y) \in \mathbb{Z}^2 \setminus \{(0,0)\}$.  Then $(x,y)$ is visible if and only if $\gcd(x,y)=1$.
\end{proposition}

\begin{proof}
Let $(x,y)$ be a non-origin point in $\mathbb{Z}^2$.  Suppose $d = \gcd(x,y)$.  If $d>1$ then $\left( \frac{x}{d}, \frac{y}{d} \right)$ lies strictly between the points $(0,0)$ and $(x,y)$, and hence $(x,y)$ is not visible.  Thus $(x,y)$ visible implies that $\gcd(x,y) = 1$.

Conversely, assume that $d=1$ and suppose by way of contradiction that $(x,y)$ is not visible from the origin.  Then there is a point $(x_0,y_0) \in \mathbb{Z}^2$ such that $(x,y) = (cx_0,cy_0)$ for some integer $c>1$.  That is, $c$ divides both $x$ and $y$.  But $d=1$ is the greatest common divisor of $x$ and $y$, contradicting that $c>1$.  Thus if $\gcd(x,y)=1$ then $(x,y)$ is visible.
\end{proof}

Now that we have a simple criterion for an integer lattice point's visibility, it is natural to inquire what fraction of integer lattice points are visible from the origin.  That is, we ask:
\begin{center}
\textit{What is the density of visible lattice points in $\mathbb{Z}^2$?}
\end{center}
Let $T(n)$ equal the total number of integer lattice points in an $n \times n$ square centered at the origin, and let $V(n)$ equal the number of these points visible from the origin.  Then it suffices to compute the limit of $\frac{V(n)}{T(n)}$ as $n$ approaches infinity.  It turns out this limit is $\frac{6}{\pi^2}$.  Proofs of this famous result are well known with the earliest proofs given in the late 19th century (see references in Remark~\ref{rem:historical_background}).  Many modern solutions involve the M\"obius inversion formula and Euler's totient function.  In Proposition~\ref{prop:fraction_of_visible_points}, we provide an alternative proof that is essentially an application of Euler's famous product formula and utilizes the number-theoretic criterion for the visibility of a lattice point given in Proposition~\ref{prop:visibility_of_a_point}.

\begin{remark}[Historical background to the problem]\label{rem:historical_background}
The historical record of the original authorship of the result in Proposition~\ref{prop:fraction_of_visible_points} is inaccurately described on a number of occasions in the literature.  Originally, the question on the probability of two random integers being coprime was raised in 1881 by Ces\`aro~\cite{Ces1881}.  Two years later, he and Sylvester independently proved the result~\cite{Ces1883} and \cite{Syl1883}, respectively.  Earlier in 1849, Dirichlet proved a slightly weaker form of the result~\cite{Dir1849}.  The generalization to $k$ coprime integers with $k>2$ was presented again by Ces\`aro in 1884~\cite{Ces1884}.  This result was apparently proven independently in 1900 by Lehmer~\cite{Leh1900}.  
\end{remark}

\begin{remark}
Since there is no uniform distribution on the natural numbers, it is somewhat imprecise to speak about the probability that two integers chosen at random are relatively prime.  However, if we consider the uniform distribution on the set $\{1,2, \ldots, n \}$ and take the limit as $n$ approaches infinity, then it is within this context that we make any probability statements in Proposition~\ref{prop:fraction_of_visible_points}.
\end{remark}

\begin{proposition}\label{prop:fraction_of_visible_points}
The density of integer lattice points that are visible from the origin is $\frac{6}{\pi^2}$, or approximately $60\%$.
\end{proposition}

\begin{proof}
It suffices to show that a lattice point chosen at random has a probability of $\frac{6}{\pi^2}$ of being visible from the origin.  Let $x$ and $y$ be randomly selected integers.  Recall that $(x,y)$ is visible if and only if $\gcd(x,y)=1$ by Proposition~\ref{prop:visibility_of_a_point}.  Hence it suffices to compute the probability that no prime divides both $x$ and $y$.  The probability that $x$ is divisible by the prime $p$ is $\frac{1}{p}$.  Similarly $y$ is divisible by $p$ with probability $\frac{1}{p}$.  By mutual independence, the probability that both $x$ and $y$ are divisible by $p$ is $\frac{1}{p^2}$.  Hence, the probability that both integers $x$ and $y$ are not divisible by $p$ is $1 - \frac{1}{p^2}$.  For distinct primes, these divisibility events are mutually independent, thus the probability that no prime divides both $x$ and $y$ is the following product over the primes:
$$\displaystyle\prod_p \left( 1 - \frac{1}{p^2} \right).$$
To calculate this infinite product, it is helpful to consider the Riemann zeta function
$$\zeta(s) = \sum_{n\geq 1} \frac{1}{n^s} = \frac{1}{1^s} + \frac{1}{2^s} + \frac{1}{3^s} + \cdots$$
for $s>1$.  A result of Euler connects this infinite sum with an infinite product of infinite sums over the primes.  The essence of Euler's proof is his use of the fundamental theorem of arithmetic to observe that the sum $\zeta(s)$ can be written as the following infinite product
\begin{align}
\sum_{n\geq 1} \frac{1}{n^s} = \prod_p \left( 1 + \frac{1}{p^{s}} + \frac{1}{p^{2s}} + \frac{1}{p^{3s}} + \cdots \right).\label{eqn:Euler_product}
\end{align}
To prove Equation~(\ref{eqn:Euler_product}), Euler observed that since each $n$ in the denominator on the left-hand side is of the form $n = p_{i_1}^{\alpha_{i_1}}p_{i_2}^{\alpha_{i_2}}\cdots p_{i_k}^{\alpha_{i_k}}$ for some $k$ by the fundamental theorem of arithmetic, then by multiplying out the product on the right-hand side, each term $\frac{1}{n^s}$ on the left-hand side appears exactly once, as a product of the appropriate powers of the primes in $n$.  And since each multiplicand on the right-hand side is a geometric series of the form $\dfrac{1}{1-\frac{1}{p^s}}$, Equation~(\ref{eqn:Euler_product}) becomes
$$\sum_{n\geq 1} \frac{1}{n^s} = \prod_p \dfrac{1}{1-\frac{1}{p^s}}.$$
Setting $s=2$ and taking reciprocals, we get
\begin{align*}
\zeta(2)^{-1} = \dfrac{1}{\sum_{n\geq 1} \frac{1}{n^2}} = \prod_p \left( 1 - \frac{1}{p^2} \right),
\end{align*}
where the right-hand side is the probability value we seek, and the left-hand side is the reciprocal of the well-known evaluation of the Riemann zeta function at $s=2$, namely $\zeta(2) = \frac{\pi^2}{6}$ (the solution to the famous \textit{Basel Problem}\footnote{The Basel Problem asks for the exact sum of the reciprocals of the squares of the positive integers.  There are a variety of proofs of this result.  Chapman in 2003 gives the details of 14 different proofs~\cite{Chap03}.  More recently in 2015, Moreno compiles a comprehensive list of 85 references from Euler to the present that address the Basel Problem~\cite{Moreno15}.}).  Hence the fraction of lattice points $(x,y)$ visible from the origin is $\frac{6}{\pi^2}$ (or approximately 60\%), as desired.
\end{proof}


\section{The traditional method to find hidden forests}\label{sec:the_CRT_algorithm}

In the previous section we showed that approximately $60\%$ of the integer lattice is visible, and hence approximately $40\%$ lies hidden from view.  In this section, we find arbitrarily large patches of hidden square regions in $\mathbb{Z}^2$ using the known method.  This technique is what we call the \texttt{CRT}-algorithm, since it is mainly an application of the Chinese Remainder Theorem (CRT).  The strategy is to find two sets of $n$ consecutive integers
$$\mathbf{X} = \{x_1, x_2, \ldots, x_n\} \mbox{ and } \mathbf{Y} = \{y_1, y_2, \ldots, y_n\}$$
such that $\mathbf{X} \cap \mathbf{Y} = \emptyset$ and $\gcd(x_i,y_j) > 1$ for all $1 \leq i,j \leq n$.  Then it is clear that the $n^2$ points in the set $\{(x_i, y_j) \; | \; 1 \leq i,j \leq n\}$ yield the desired hidden square region.  To this end, we first establish some necessary preliminary definitions.

\begin{definition}[Hidden forest]
An $n \times n$ \textit{hidden forest} in $\mathbb{Z}^2$ is a square patch of $n^2$ adjacent invisible integer lattice points.  We denote this hidden forest by the symbol $\Hforest{n}$ where $(x,y)$ is the closest corner lattice point of the square to the origin.  By the remark below, this closest corner point is well-defined.
\end{definition}

\begin{remark}\label{rem:no_axes}
Observe that the points $(x,\pm 1)$ and $(\pm 1,y)$ are visible for all $x,y \in \mathbb{Z}$ by Proposition~\ref{prop:visibility_of_a_point}.  Hence no nontrivial (that is, $n>1$) hidden forest $\Hforest{n}$ will contain any points on the $x$- or $y$-axes.  Hence we conclude that any $\Hforest{n}$ for $n>1$ is completely contained in the interior of one of the four quadrants.
\end{remark}

\begin{definition}[Prime matrix]\label{def:prime_matrix}
Let $\{p_1, p_2, \ldots, p_{n^2} \}$ be the set of the first $n^2$ primes.  Construct an $n \times n$ matrix with these primes by filling row $i$ with $p_{(i-1)n+1}$ through $p_{(i-1)n+n}$ for each $1 \leq i \leq n$ to yield the following:

$$\left(
\begin{array}{cccccc}
p_1 & p_2 & \cdots & p_j & \cdots & p_n \\
p_{n+1} & p_{n+2} & \cdots & p_{n+j} & \cdots & p_{2n} \\
p_{2n+1} & p_{2n+2} & \cdots & p_{2n+j} & \cdots & p_{3n}\\
\vdots & \vdots & & \vdots & & \vdots\\
p_{(i-1)n+1} & p_{(i-1)n+2} & \cdots & \framebox[1.1\width]{\mbox{$p_{(i-1)n+j}$}} & \cdots & p_{(i-1)n+n}\\
\vdots & \vdots & & \vdots & & \vdots\\
p_{(n-1)n+1} & p_{(n-1)n+2} & \cdots & p_{(n-1)n+j} & \cdots & p_{n^2}
\end{array} \right).
$$
Note that the prime $p_{(i-1)n+j}$, boxed for visual ease, is located in row $i$ and column $j$ of the matrix.  We call this $n \times n$ matrix a \textit{prime matrix} and denote it $P_n$.
\end{definition}

\subsection{The \texttt{CRT}-algorithm}\label{subsec:known_method}

The following theorem is the primary tool used in the \texttt{CRT}-algorithm to find hidden forests of arbitrary size.

\begin{theorem}\label{thm:2by2_case}
For each $n \in \mathbb{N}$, there exist two sets of $n$ consecutive numbers $\mathbf{X}=\{x_1, x_2, \ldots, x_n\}$ and $\mathbf{Y}=\{y_1, y_2, \ldots, y_n\}$ such that $\mathbf{X} \cap \mathbf{Y} = \emptyset$ and $\gcd(x_i,y_j) > 1$ for all $1 \leq i,j \leq n$.
\end{theorem}

\begin{proof}
Fix $n \in \mathbb{N}$.  Consider the prime matrix $P_n$.  Let $R_i$ and $C_j$ be the product of the entries in row $i$ and column $j$, respectively, so we have
$$R_i = \prod_{k=1}^n p_{(i-1)n+k} \;\;\;\; \mbox{and} \;\;\;\; C_j = \prod_{k=0}^{n-1} p_{kn+j}.$$
Since they share no primes in common, the row products $R_1, R_2, \ldots, R_n$ are pairwise relatively prime.  Similarly, the column products $C_1, C_2, \ldots, C_n$ are pairwise relatively prime.  Consider the following pair of systems of linear congruences:
$$
\begin{cases}
x+1 \equiv 0 \pmod{R_1}\\
x+2 \equiv 0 \pmod{R_2}\\
\hspace{.7in}\vdots & \\
x+n \equiv 0 \pmod{R_n}
\end{cases}
\hspace{.6in}
\begin{cases}
y+1 \equiv 0 \pmod{C_1}\\
y+2 \equiv 0 \pmod{C_2}\\
\hspace{.7in} \vdots & \\
y+n \equiv 0 \pmod{C_n}.
\end{cases}
$$
Observe that $R_1 \cdot R_2 \cdots R_n = C_1 \cdot C_2 \cdots C_n = \prod_{i=1}^{n^2} p_i$, which we denote $M$.  By the Chinese Remainder Theorem, there exist solutions $x_0$ and $y_0$ to the left and right systems, respectively, such that $x_0$ and $y_0$ are unique modulo $M$.   Define the set $\mathbf{X} = \{x_0 + 1, x_0 + 2, \ldots, x_0 + n \}$ and the set $\mathbf{Y} = \{y_0 + 1, y_0 + 2, \ldots, y_0 + n \}$.  We claim that none of the integers in $\mathbf{X}$ are pairwise relatively prime to any of the integers in $\mathbf{Y}$.  For an arbitrary $x_0 + i \in \mathbf{X}$ and $y_0 + j \in \mathbf{Y}$, these two elements by construction are multiples of $R_i$ and $C_j$, respectively. Hence the prime that lies in the intersection of row $i$ and column $j$ in the matrix, namely $p_{(i-1)n+j}$, divides $\gcd(x_0+i,y_0+j)$.  Thus $\gcd(x_0+i,y_0+j)>1$ as desired.

Observe that for $n \geq 2$ the sets $\mathbf{X}$ and $\mathbf{Y}$ are necessarily disjoint. Otherwise if $\mathbf{X} \cap \mathbf{Y} \neq \emptyset$ then some element $a \in \mathbf{X}$ is relatively prime to some element $a\pm 1 \in \mathbf{Y}$ since $\gcd(a, a\pm 1) = 1$, contradicting $\gcd(x_0+i,y_0+j)>1$ for all $1 \leq i,j \leq n$. For the trivial case when $n=1$, the algorithm above yields $\mathbf{X}=\mathbf{Y}=\{2\}$. So set $\mathbf{Y}=\{4\}$ and hence $\mathbf{X} \cap \mathbf{Y} = \emptyset$.
\end{proof}

\bigskip

\boitegrise{
\noindent\textbf{The \texttt{CRT}-algorithm to construct a hidden forest $\Hforest{n}$:}
\begin{enumerate}
\item Fix a value $n \in \mathbb{N}$.
\item Construct the prime matrix $P_n$.
\item Apply Theorem~\ref{thm:2by2_case} to $P_n$ to yield sets $\mathbf{X}$ and $\mathbf{Y}$.
\item Construct the hidden forest $\Hforest{n}$ from $\mathbf{X}$ and $\mathbf{Y}$.
\end{enumerate}\vspace{-.2in}}{0.65\textwidth}

\bigskip

\subsection{The gcd-grid and gcd-matrix yielded by prime matrices}\label{subsec:gcd-grid_and_matrix}

By Theorem~\ref{thm:2by2_case}, the prime matrix $P_n$ yields the hidden forest $\Hforest{n}$ comprised of the $n^2$ points $(x_i,y_j)$ where $x_i = x_0 + i \in \mathbf{X}$ and $y_j = y_0 + j \in \mathbf{Y}$ for $1 \leq i,j \leq n$. The forest $\Hforest{n}$ is shown in Figure~\ref{fig:generic_hidden_forest}.  For each $\Hforest{n}$ we can write a corresponding $n \times n$ array of numbers called the \textit{gcd-grid} where $g_{i,j} = \gcd(x_i,y_j)$ for $1 \leq i,j \leq n$.  The gcd-grid is shown in Figure~\ref{fig:generic_gcd_grid}.

\begin{figure}[H]
\begin{center}
\begin{subfigure}[h]{0.45\textwidth}
		\begin{tikzpicture}[scale=.8]
		\draw (0,0)--(0,2)--(2,2)--(2,0)--(0,0);
		\draw (0,1)--(2,1); \draw (1,0)--(1,2); \draw(0,3.5)--(2,3.5); \draw(0,5)--(2,5);
		\draw(3,0)--(3,2);
		\draw(4,0)--(4,2)--(5,2)--(5,0)--(4,0); \draw(4,1)--(5,1);
		\draw(4,3.5)--(5,3.5); \draw(4,5)--(5,5);
		\foreach \x in {0,...,5}
				\draw (\x,5)--(\x,4.75);
		\foreach \x in {0,...,5}
				\draw(\x,3.5)--(\x,3.75);
		\foreach \x in {0,...,5}
				\draw(\x,3.5)--(\x,3.25);
		\foreach \x in {0,...,5}
				\draw(\x,2)--(\x,2.25);
		\foreach \x in {0,...,5}
				\node at (\x,4.35) {$\vdots$};
		\foreach \x in {0,...,5}
				\node at (\x,2.85) {$\vdots$};
		\node at (2.5,3.5) {$\cdots$};
		\node at (3.5,3.5) {$\cdots$};
		\node at (2.5,0) {$\cdots$};
		\node at (3.5,0) {$\cdots$};
		\node at (2.5,1) {$\cdots$};
		\node at (3.5,1) {$\cdots$};
		\node at (2.5,2) {$\cdots$};
		\node at (3.5,2) {$\cdots$};
		\node at (2.5,5) {$\cdots$};
		\node at (3.5,5) {$\cdots$};
		\foreach \x in {0,...,5}
				\shade[ball color=blue] (\x,0) circle (.65ex);
		\foreach \x in {0,...,5}
				\shade[ball color=blue] (\x,1) circle (.65ex);
		\foreach \x in {0,...,5}
				\shade[ball color=blue] (\x,2) circle (.65ex);
		\foreach \x in {0,...,5}
				\shade[ball color=blue] (\x,3.5) circle (.65ex);
		\foreach \x in {0,...,5}
				\shade[ball color=blue] (\x,5) circle (.65ex);
		\shade[ball color=red] (3,3.5) circle (.65ex);
		\node [below left] at (0,0) {\footnotesize $(x_1,y_1)$};
		\node [left] at (0,1) {\footnotesize $(x_1,y_2)$};
		\node [left] at (0,3.5) {\footnotesize $(x_1,y_j)$};
		\node [above left] at (0,5) {\footnotesize $(x_1,y_n)$};
		\node [below] at (1,0) {\footnotesize $(x_2,y_1)$};
		\node [below] at (3,0) {\footnotesize $(x_i,y_1)$};
		\node [below right] at (5,0) {\footnotesize $(x_n,y_1)$};
		\node [above, red] at (3,3.5) {\footnotesize $(x_i,y_j)$};
		\node [above] at (3,5) {\footnotesize $(x_i,y_n)$};
		\node [right] at (5,3.5) {\footnotesize $(x_n,y_j)$};
		\node [above right] at (5,5) {\footnotesize $(x_n,y_n)$};
		\end{tikzpicture}
\caption{Hidden forest $\Hforest{n}$}
\label{fig:generic_hidden_forest}
\end{subfigure}
\begin{subfigure}[h]{0.45\textwidth}
\hspace{.5in}
		\begin{tikzpicture}[scale=.8]
		\draw (0,0)--(0,2)--(2,2)--(2,0)--(0,0);
		\draw (0,1)--(2,1); \draw (1,0)--(1,2); \draw(0,3.5)--(2,3.5); \draw(0,5)--(2,5);
		\draw(3,0)--(3,2);
		\draw(4,0)--(4,2)--(5,2)--(5,0)--(4,0); \draw(4,1)--(5,1);
		\draw(4,3.5)--(5,3.5); \draw(4,5)--(5,5);
		\foreach \x in {0,...,5}
				\draw (\x,5)--(\x,4.75);
		\foreach \x in {0,...,5}
				\draw(\x,3.5)--(\x,3.75);
		\foreach \x in {0,...,5}
				\draw(\x,3.5)--(\x,3.25);
		\foreach \x in {0,...,5}
				\draw(\x,2)--(\x,2.25);
		\foreach \x in {0,...,5}
				\node at (\x,4.35) {$\vdots$};
		\foreach \x in {0,...,5}
				\node at (\x,2.85) {$\vdots$};
		\node at (2.5,3.5) {$\cdots$};
		\node at (3.5,3.5) {$\cdots$};
		\node at (2.5,0) {$\cdots$};
		\node at (3.5,0) {$\cdots$};
		\node at (2.5,1) {$\cdots$};
		\node at (3.5,1) {$\cdots$};
		\node at (2.5,2) {$\cdots$};
		\node at (3.5,2) {$\cdots$};
		\node at (2.5,5) {$\cdots$};
		\node at (3.5,5) {$\cdots$};
		\foreach \x in {0,...,5}
				\shade[ball color=blue] (\x,0) circle (.65ex);
		\foreach \x in {0,...,5}
				\shade[ball color=blue] (\x,1) circle (.65ex);
		\foreach \x in {0,...,5}
				\shade[ball color=blue] (\x,2) circle (.65ex);
		\foreach \x in {0,...,5}
				\shade[ball color=blue] (\x,3.5) circle (.65ex);
		\foreach \x in {0,...,5}
				\shade[ball color=blue] (\x,5) circle (.65ex);
		\shade[ball color=red] (3,3.5) circle (.65ex);
		\node [below left] at (0,0) {\footnotesize $g_{1,1}$};
		\node [left] at (0,1) {\footnotesize $g_{1,2}$};
		\node [left] at (0,3.5) {\footnotesize $g_{1,j}$};
		\node [above left] at (0,5) {\footnotesize $g_{1,n}$};
		\node [below] at (1,0) {\footnotesize $g_{2,1}$};
		\node [below] at (3,0) {\footnotesize $g_{i,1}$};
		\node [below right] at (5,0) {\footnotesize $g_{n,1}$};
		\node [above, red] at (3,3.5) {\footnotesize $g_{i,j}$};
		\node [above] at (3,5) {\footnotesize $g_{i,n}$};
		\node [right] at (5,3.5) {\footnotesize $g_{n,j}$};
		\node [above right] at (5,5) {\footnotesize $g_{n,n}$};
		\end{tikzpicture}
\caption{The $\gcd$-grid of $\Hforest{n}$}
\label{fig:generic_gcd_grid}
\end{subfigure}
\end{center}
\vspace{-.15in}
\caption{}
\end{figure}

We may also consider the $\gcd$-grid as a matrix if we collapse the grid structure and place the $n^2$ $\gcd$-values into a matrix in the same locations that they appear in the $\gcd$-grid as follows:
$$\mathrm{Gcd}_{P_n} = \left(
\begin{array}{cccccc}
g_{1,n} & g_{2,n} & \cdots & g_{i,n} & \cdots & g_{n,n} \\
\vdots & \vdots &  & \vdots & & \vdots\\
g_{1,j} & g_{2,j} & \cdots & g_{i,j} & \cdots & g_{n,j}\\
\vdots & \vdots &  & \vdots & & \vdots\\
g_{1,2} & g_{2,2} & \cdots & g_{i,2} & \cdots & g_{n,2}\\
g_{1,1} & g_{2,1} & \cdots & g_{i,1} & \cdots & g_{n,1} 
\end{array}
\right).
$$
We call this matrix arising from the $\gcd$-grid the \textit{gcd-matrix} corresponding to $P_n$ and denote it $\mathrm{Gcd}_{P_n}$.  If we denote the prime $p_{(i-1)n+j}$ in row $i$ and column $j$ of matrix $P_n$ as $p_{i,j}$, then the prime matrix given in Definition~\ref{def:prime_matrix} can be written as follows:
$$P_n = \left(
\begin{array}{cccccc}
p_{1,1} & p_{1,2} & \cdots & p_{1,j} & \cdots & p_{1,n} \\
p_{2,1} & p_{2,2} & \cdots & p_{2,j} & \cdots & p_{2,n} \\
\vdots & \vdots &  & \vdots & & \vdots\\
p_{i,1} & p_{i,2} & \cdots & p_{i,j} & \cdots & p_{i,n}\\
\vdots & \vdots &  & \vdots & & \vdots\\
p_{n,1} & p_{n,2} & \cdots & p_{n,j} & \cdots & p_{n,n} 
\end{array}
\right).
$$

\boitegrise{
\begin{remark} The $(i,j)$-entry of $P_n$ is $p_{i,j}$.  However, the $(i,j)$-entry of $\mathrm{Gcd}_{P_n}$ is not $g_{i,j}$.  In fact, the entry $g_{i,j}$ is in row $n-(j-1)$ and column $i$ of $\mathrm{Gcd}_{P_n}$.\vspace{-.3in}\end{remark}}{0.85\textwidth}

Comparing the locations of the entries $g_{i,j}$ and $p_{i,j}$ of the matrices $\mathrm{Gcd}_{P_n}$ and $P_n$, respectively, as the values $i$ and $j$ vary, we observe that the subscripts of the entries in one matrix are a rotation of the subscripts of the entries in the other.  In particular, the following proposition describes the relationship between the matrices $\mathrm{Gcd}_{P_n}$ and $P_n$ via a third matrix which we call $\widetilde{\mathrm{Gcd}}_{P_n}$.

\begin{proposition}\label{prop:correspondence_betw_P_and_GcdP}
Let $P_n$ be a prime matrix.  A rotation by $90^\circ$ counter-clockwise of the entries in $P_n$ gives a corresponding matrix which we denote by $\widetilde{\mathrm{Gcd}}_{P_n}$, and the $(i,j)$-entry in $\widetilde{\mathrm{Gcd}}_{P_n}$ divides the $(i,j)$-entry in the $\gcd$-matrix $\mathrm{Gcd}_{P_n}$.

\end{proposition}

\begin{proof}
The rotational relationship between $P_n$ and $\widetilde{\mathrm{Gcd}}_{P_n}$ is given by a simple matrix calculation.  If we let $\mathrm{AD}_n$ be the anti-diagonal matrix---that is, a matrix with ones in the anti-diagonal and zeroes elsewhere, then $\widetilde{\mathrm{Gcd}}_{P_n} = (P_n\cdot\mathrm{AD}_n)^T$, where $T$ denotes the transpose of a matrix.  In particular, multiplying $P_n$ on the right by $\mathrm{AD}_n$ reverses the columns of $P_n$, and then transposing this result yields $\widetilde{\mathrm{Gcd}}_{P_n}$ as desired.  After this rotation on $P_n$ is performed, the entry $p_{i,j}$ of $\widetilde{\mathrm{Gcd}}_{P_n}$ is then located in row $n-(j-1)$ and column $i$.  In this same location in $\mathrm{Gcd}_{P_n}$ is $g_{i,j}$. Below we give an illustration of this process.
$$P_n = \left(
\begin{array}{cccccc}
p_{1,1} & p_{1,2} & \cdots & p_{1,j} & \cdots & p_{1,n} \\
p_{2,1} & p_{2,2} & \cdots & p_{2,j} & \cdots & p_{2,n} \\
\vdots & \vdots &  & \vdots & & \vdots\\
p_{i,1} & p_{i,2} & \cdots & p_{i,j} & \cdots & p_{i,n}\\
\vdots & \vdots &  & \vdots & & \vdots\\
p_{n,1} & p_{n,2} & \cdots & p_{n,j} & \cdots & p_{n,n} 
\end{array}
\right)
\;\; \xmapsto[90^\circ \; \mathrm{left}]{\RotateLeft} \;\;
\widetilde{\mathrm{Gcd}}_{P_n} = \left(
\begin{array}{cccccc}
p_{1,n} & p_{2,n} & \cdots & p_{i,n} & \cdots & p_{n,n} \\
\vdots & \vdots &  & \vdots & & \vdots\\
p_{1,j} & p_{2,j} & \cdots & p_{i,j} & \cdots & p_{n,j}\\
\vdots & \vdots &  & \vdots & & \vdots\\
p_{1,2} & p_{2,2} & \cdots & p_{i,2} & \cdots & p_{n,2}\\
p_{1,1} & p_{2,1} & \cdots & p_{i,1} & \cdots & p_{n,1}
\end{array}
\right).
$$

In the proof of Theorem~\ref{thm:2by2_case}, we observed that by construction the prime $p_{i,j}$ divides the value $\gcd(x_0 + i, y_0 + j) = g_{i,j}$.  Hence, the $(i,j)$-entry in $\widetilde{\mathrm{Gcd}}_{P_n}$ divides the $(i,j)$-entry in $\mathrm{Gcd}_{P_n}$.
\end{proof}

\begin{remark}
The rotational relationship between $P_n$ and $\widetilde{\mathrm{Gcd}}_{P_n}$ proves to be very important in Section~\ref{sec:quasiprime} when we perform the reverse rotation.  Starting from a $\gcd$-matrix, a \textit{clockwise} rotation will help us produce a \textit{quasiprime} matrix, crucial for finding closer hidden forests.
\end{remark}

\subsection{An application: the \texorpdfstring{$n=2,3,4$}{n=2,3,4} cases}\label{subsec:applications_of_known_method}

\begin{example}\label{exam:2by2_hidden_forest}
In the $2 \times 2$ case, using Theorem~\ref{thm:2by2_case}, we set $n=2$ and the prime matrix is
$$P_2 = \left(
\begin{array}{cc}
2 & 3 \\
5 & 7 \end{array}
\right).$$
The row products are $R_1 = 6$ and $R_2 = 35$, while the column products are $C_1 = 10$ and $C_2 = 21$.  Hence the corresponding linear congruences we need to solve are
$$\begin{array}{lcl}
x+1 \equiv 0 \pmod{6} & \hspace{1in} & y+1 \equiv 0 \pmod{10}\\
x+2 \equiv 0 \pmod{35} &  \hspace{1in} & y+2 \equiv 0 \pmod{21}.
\end{array}$$
By the \texttt{CRT}-algorithm, the left and right systems have the unique solutions $x_0 = 173 \pmod{210}$ and $y_0 = 19 \pmod{210}$, respectively.  Set $\mathbf{X} = \{174,175\}$ and $\mathbf{Y} = \{20,21\}$.  Then $\mathbf{X} \cap \mathbf{Y} = \emptyset$ and $\gcd(x_i,y_j)>1$ for all $1 \leq i,j \leq 2$.  Thus there is a hidden forest $H_{(174,20)}^2$ of four trees at $(174,20)$, $(174,21)$, $(175,20)$, and $(175,21)$.  In the figure below we give the hidden forest on the left and its corresponding $\gcd$-grid on the right.
\begin{center}
\begin{tikzpicture}[scale=.75]
\draw (0,0)--(0,1)--(1,1)--(1,0)--(0,0);
\foreach \x in {0,1}
		\shade[ball color=blue] (\x,0) circle (.75ex);
\foreach \x in {0,1}
		\shade[ball color=blue] (\x,1) circle (.75ex);
\node [left] at (0,0) {\footnotesize (174,20)};
\node [left] at (0,1) {\footnotesize (174,21)};
\node [right] at (1,0) {\footnotesize (175,20)};
\node [right] at (1,1) {\footnotesize (175,21)};
\end{tikzpicture}
\hspace{1in}
\begin{tikzpicture}[scale=.75]
\draw (0,0)--(0,1)--(1,1)--(1,0)--(0,0);
\foreach \x in {0,1}
		\shade[ball color=blue] (\x,0) circle (.75ex);
\foreach \x in {0,1}
		\shade[ball color=blue] (\x,1) circle (.75ex);
\node [left] at (0,0) {\footnotesize $\gcd(174,20)=2$};
\node [left] at (0,1) {\footnotesize $\gcd(174,21)=3$};
\node [right] at (1,0) {\footnotesize $5=\gcd(175,20)$};
\node [right] at (1,1) {\footnotesize $7=\gcd(175,21)$};
\end{tikzpicture}
\end{center}
Then by Proposition~\ref{prop:correspondence_betw_P_and_GcdP}, we have the following map from $P_2$ to $\widetilde{\mathrm{Gcd}}_{P_2}$:
$$P_2 = \left(
\begin{array}{cc}
2 & 3 \\
5 & 7 \end{array}
\right)
\;\; \xmapsto[90^\circ \; \mathrm{left}]{\RotateLeft} \;\;
\widetilde{\mathrm{Gcd}}_{P_2} = \left(
\begin{array}{cc}
3 & 7 \\
2 & 5 \end{array}
\right).
$$
In this example, the $\widetilde{\mathrm{Gcd}}_{P_2}$ coincides with the $\gcd$-matrix $\mathrm{Gcd}_{P_2}$, and so the $(i,j)$-entry of $\widetilde{\mathrm{Gcd}}_{P_2}$ divides the $(i,j)$-entry of $\mathrm{Gcd}_{P_2}$ as Proposition~\ref{prop:correspondence_betw_P_and_GcdP} guarantees.  This is an effect of the so-called \textit{law of small numbers}, since we see in the larger $n$ cases to follow that this coincidence does not occur.
\end{example}

\begin{example}\label{exam:3by3_hidden_forest}
In the $3 \times 3$ case, using Theorem~\ref{thm:2by2_case}, we set $n=3$ and the prime matrix is
$$P_3 = \left(
\begin{array}{ccc}
2 & 3 & 5\\
7 & 11 & 13\\
17 & 19 & 23
\end{array}
\right).$$
The \texttt{CRT}-algorithm gives the solutions $x_0 = \numprint{119 740 619}$ and $y_0 = \numprint{121 379 047}$, both unique modulo $\numprint{223 092 870}$.  Hence the nine coordinates $(x_i,y_j)$ of $H_{(119 740 620,121 379 048)}^3$ have the following values and respective prime factorizations:
\begin{align*}
x_1 &= \numprint{119 740 620} = 2^2 \cdot 3 \cdot 5 \cdot \numprint{1 995 677} &\hspace{.2in} y_1 &= \numprint{121 379 048} = 2^3 \cdot 7 \cdot 17 \cdot 59 \cdot 2161\\
x_2 &= \numprint{119 740 621} = 7 \cdot 11 \cdot 13 \cdot 37 \cdot 53 \cdot 61 & \hspace{.2in} y_2 &= \numprint{121 379 049} = 3^2 \cdot 11 \cdot 19 \cdot 173 \cdot 373\\
x_3 &= \numprint{119 740 622} = 2 \cdot 17 \cdot 19 \cdot 23 \cdot 8059 &\hspace{.2in} y_3 &= \numprint{121 379 050} = 2 \cdot 5^2 \cdot 13 \cdot 23^2 \cdot 353.
\end{align*}
It is readily verified that the corresponding $3 \times 3$ hidden forest has the following $\gcd$-grid:
\begin{center}
\begin{tikzpicture}[scale=.75]
\draw (0,0)--(0,2)--(2,2)--(2,0)--(0,0);
\draw (0,1)--(2,1);
\draw (1,0)--(1,2);
\foreach \x in {0,...,2}
		\shade[ball color=blue] (\x,0) circle (.75ex);
\foreach \x in {0,...,2}
		\shade[ball color=blue] (\x,1) circle (.75ex);
\foreach \x in {0,...,2}
		\shade[ball color=blue] (\x,2) circle (.75ex);
\node [left] at (0,0) {\footnotesize $2^2$};
\node [left] at (0,1) {\footnotesize 3};
\node [left] at (0,2) {\footnotesize $2 \! \cdot \! 5$};
\node [below] at (1,0) {\footnotesize 7};
\node [below] at (1,1) {\footnotesize 11};
\node [above] at (1,2) {\footnotesize 13};
\node [right] at (2,0) {\footnotesize $2 \! \cdot \! 17$};
\node [right] at (2,1) {\footnotesize 19};
\node [right] at (2,2) {\footnotesize $2 \! \cdot \! 23$};
\end{tikzpicture}
\end{center}
For example, the top-right node corresponds to the $(x_3,y_3)$-coordinate, and it is labeled by the value $\gcd(x_3,y_3)=2 \cdot 23$ since $x_3 = 2 \cdot 17 \cdot 19 \cdot 23 \cdot 8059$ and $y_3 = 2 \cdot 5^2 \cdot 13 \cdot 23^2 \cdot 353$.  And by Proposition~\ref{prop:correspondence_betw_P_and_GcdP}, we have the following map from $P_3$ to $\widetilde{\mathrm{Gcd}}_{P_3}$:
$$P_3 = \left(
\begin{array}{ccc}
2 & 3 & 5 \\
7 & 11 & 13 \\
17 & 19 & 23
\end{array}
\right)
\;\; \xmapsto[90^\circ \; \mathrm{left}]{\RotateLeft} \;\;
\widetilde{\mathrm{Gcd}}_{P_3} = \left(
\begin{array}{ccc}
5 & 13 & 23 \\
3 & 11 & 19 \\
2 & 7 & 17
\end{array}
\right).
$$
Observe that, as expected, the $(i,j)$-entry of $\widetilde{\mathrm{Gcd}}_P$ divides the $(i,j)$-entry of the $\gcd$-matrix
$$
\mathrm{Gcd}_{P_3} = \left(
\begin{array}{ccc}
2 \! \cdot \! 5 & 13 & 2 \! \cdot \! 23 \\
3 & 11 & 19 \\
2^2 & 7 & 2 \! \cdot \! 17
\end{array}
\right).
$$
\end{example}

\begin{example}\label{exam:4by4_hidden_forest}
In the $4 \times 4$ case, the \texttt{CRT}-algorithm on the prime matrix gives the following solution:
\begin{align*}
x_0 &= \numprint{2 847 617 195 518 191 809}\\
y_0 &= \numprint{1 160 906 121 308 397 397}.
\end{align*}
\end{example}

The absurdly large solution values in Examples~\ref{exam:3by3_hidden_forest} and \ref{exam:4by4_hidden_forest} reveal that the \texttt{CRT}-algorithm applied to prime matrices is hardly useful for finding $n \times n$ hidden forests which are close to the origin for cases even as small as $n=3$ and $n=4$.  For instance, we prove later that the closest $\Hforest{3}$ is at $x=1274$ and $y=1308$.  Furthermore, we reveal that there is an $\Hforest{4}$ at $x=\numprint{134 043}$ and $y=\numprint{184 785 885}$.  The $x$-value of the $\Hforest{4}$ which the \texttt{CRT}-algorithm on a prime matrix yields is $2.12441 \times 10^{13}$ times larger than this $x$-value, \numprint{134 043}, of the $\Hforest{4}$ which we found.

\begin{remark}\label{rem:Project_Euler} It turns out that the number \numprint{134 043} is a very interesting integer; it is the smallest positive integer $n$ such that the numbers in the set $\{n, n+1, n+2, n+3\}$ have exactly four prime factors each.\footnote{This is known as Problem 47 on the website \url{https://projecteuler.net/about} started in 2001 by Colin Hughes~\cite{Somers}.  Project Euler gives a series of challenging computational problems that require more than just mathematical insights to solve. Nayuki Minase's very simple solution via \texttt{Mathematica} to Problem 47 is given in Listing~\ref{mathematica_code}.} We later use this value to calculate the closest known $\Hforest{4}$, bearing the smallest $x$-value, in Example~\ref{exam:4by4}.
\end{remark}


\section{New methods to find closer hidden forests}\label{sec:quasiprime}

The previous section detailed the well-known method of using the \texttt{CRT}-algorithm on prime matrices to find arbitrarily large hidden forests.  The main problem with that method is that for $n \geq 3$, the locations of these $\Hforest{n}$ are substantially further away from the origin and thus progressively harder to compute.  The aim of this section is to introduce two concepts, namely quasiprime matrices and strings of strongly composite integers, to help find substantially closer $\Hforest{n}$.  In this section we give the closest $\Hforest{n}$ for $n=2,3$ and the closest known (to this date) $\Hforest{n}$ for $n=4$.

Recall in Remark~\ref{rem:no_axes}, we observed that $\Hforest{n}$ can never contain points on the $x$- or $y$-axes when $n>1$, and hence each $\Hforest{n}$ lies completely within the interior of one of the four quadrants. Thus the closest corner point $(x,y)$ of $\Hforest{n}$ is well-defined up to quadrant selection.

\begin{definition}[The closest $n \times n$]\label{def:hiddenforest_distance}
A hidden forest $\Hforest{n}$ is said to have distance $d$ from the origin where $d$ is given by $d(x,y) = \sqrt{x^2 + y^2}$.  We say that $\Hforest{n}$ is the \textit{closest $n\times n$ hidden forest} if it has the minimum distance $d$ of all hidden $n\times n$ forests.
\end{definition}


\boitegrise{
\begin{convention}
In searching for the closest hidden forest it suffices to search only half of Quadrant~I.  Observe that any $\Hforest{n}$, whose lower-left corner lies in Quadrant~I and above the line $y=x$, will have seven other copies up to reflectional symmetries about the lines $y=x$, $y=-x$, the $x$-axis, and the $y$-axis (see Figure~\ref{fig:convention}).  Moreover, these seven copies are the same distance from the origin as $\Hforest{n}$ is.  So we focus only on $H^n_{(a,b)}$ in Quadrant~I such that $(a,b)$ lies above the diagonal $y=x$ (that is, $a < b$).  Note that $H^n_{(a,a)}$ can never exist if $n>1$ since $\gcd(a,a+1)=1$ and hence $(a,a+1)$ is a visible point.\vspace{-.3in}
\end{convention}}{0.9\textwidth}

\begin{figure}[H]
\begin{center}
\begin{tikzpicture}[scale=.2]
\draw [<->, ultra thick, blue] (22,0) -- (0,0) -- (0,22);
\draw [<->, ultra thick, blue] (-22,0) -- (0,0) -- (0,-22);
\node at (0,0) {\includegraphics[width=8ex]{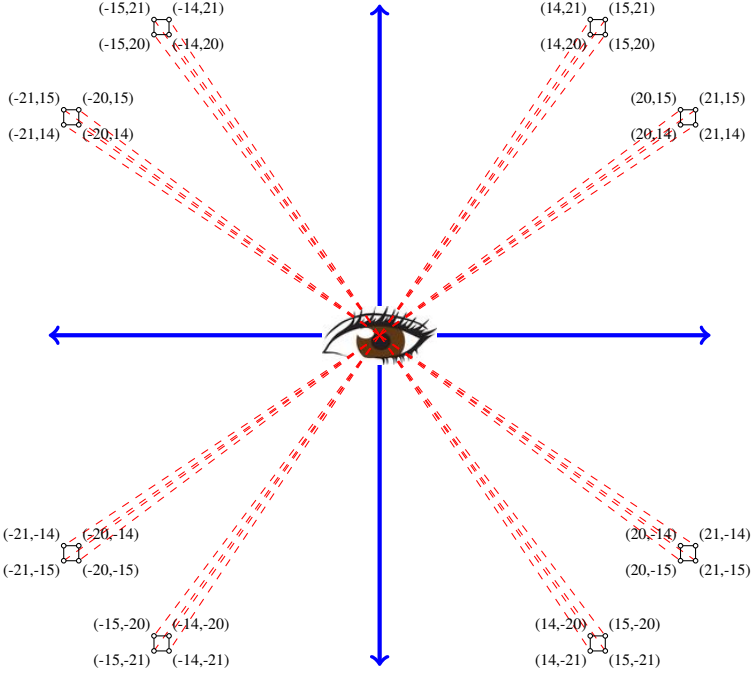}};
\draw[dashed,red] (0,0) -- (20,14);
\draw[dashed,red] (0,0) -- (21,14);
\draw[dashed,red] (0,0) -- (20,15);
\draw[dashed,red] (0,0) -- (21,15);
\draw[thin] (20,14) rectangle (21,15);
\node [below left] at (20.5,14.5) {\tiny(20,14)};
\node [below right] at (20.5,14.5) {\tiny(21,14)};
\node [above left] at (20.5,14.5) {\tiny(20,15)};
\node [above right] at (20.5,14.5) {\tiny(21,15)};
\draw plot[only marks, mark=*, mark options={fill=white}, mark size=.8ex] coordinates {(20,14) (21,14) (20,15) (21,15)};
\draw[dashed,red] (0,0) -- (14,20);
\draw[dashed,red] (0,0) -- (15,20);
\draw[dashed,red] (0,0) -- (14,21);
\draw[dashed,red] (0,0) -- (15,21);
\draw[thin] (14,20) rectangle (15,21);
\node [below left] at (14.5,20.5) {\tiny(14,20)};
\node [below right] at (14.5,20.5) {\tiny(15,20)};
\node [above left] at (14.5,20.5) {\tiny(14,21)};
\node [above right] at (14.5,20.5) {\tiny(15,21)};
\draw plot[only marks, mark=*, mark options={fill=white}, mark size=.8ex] coordinates {(14,20) (15,20) (14,21) (15,21)};
\draw[dashed,red] (0,0) -- (20,-14);
\draw[dashed,red] (0,0) -- (21,-14);
\draw[dashed,red] (0,0) -- (20,-15);
\draw[dashed,red] (0,0) -- (21,-15);
\draw (20,-14) rectangle (21,-15);
\node [below left] at (20.5,-14.5) {\tiny(20,-15)};
\node [below right] at (20.5,-14.5) {\tiny(21,-15)};
\node [above left] at (20.5,-14.5) {\tiny(20,-14)};
\node [above right] at (20.5,-14.5) {\tiny(21,-14)};
\draw plot[only marks, mark=*, mark options={fill=white}, mark size=.8ex] coordinates {(20,-14) (21,-14) (20,-15) (21,-15)};
\draw[dashed,red] (0,0) -- (14,-20);
\draw[dashed,red] (0,0) -- (15,-20);
\draw[dashed,red] (0,0) -- (14,-21);
\draw[dashed,red] (0,0) -- (15,-21);
\draw (14,-20) rectangle (15,-21);
\node [below left] at (14.5,-20.5) {\tiny(14,-21)};
\node [below right] at (14.5,-20.5) {\tiny(15,-21)};
\node [above left] at (14.5,-20.5) {\tiny(14,-20)};
\node [above right] at (14.5,-20.5) {\tiny(15,-20)};
\draw plot[only marks, mark=*, mark options={fill=white}, mark size=.8ex] coordinates {(14,-20) (15,-20) (14,-21) (15,-21)};
\draw[dashed,red] (0,0) -- (-20,14);
\draw[dashed,red] (0,0) -- (-21,14);
\draw[dashed,red] (0,0) -- (-20,15);
\draw[dashed,red] (0,0) -- (-21,15);
\draw[thin] (-20,14) rectangle (-21,15);
\node [below left] at (-20.5,14.5) {\tiny(-21,14)};
\node [below right] at (-20.5,14.5) {\tiny(-20,14)};
\node [above left] at (-20.5,14.5) {\tiny(-21,15)};
\node [above right] at (-20.5,14.5) {\tiny(-20,15)};
\draw plot[only marks, mark=*, mark options={fill=white}, mark size=.8ex] coordinates {(-20,14) (-21,14) (-20,15) (-21,15)};
\draw[dashed,red] (0,0) -- (-14,20);
\draw[dashed,red] (0,0) -- (-15,20);
\draw[dashed,red] (0,0) -- (-14,21);
\draw[dashed,red] (0,0) -- (-15,21);
\draw[thin] (-14,20) rectangle (-15,21);
\node [below left] at (-14.5,20.5) {\tiny(-15,20)};
\node [below right] at (-14.5,20.5) {\tiny(-14,20)};
\node [above left] at (-14.5,20.5) {\tiny(-15,21)};
\node [above right] at (-14.5,20.5) {\tiny(-14,21)};
\draw plot[only marks, mark=*, mark options={fill=white}, mark size=.8ex] coordinates {(-14,20) (-15,20) (-14,21) (-15,21)};
\draw[dashed,red] (0,0) -- (-20,-14);
\draw[dashed,red] (0,0) -- (-21,-14);
\draw[dashed,red] (0,0) -- (-20,-15);
\draw[dashed,red] (0,0) -- (-21,-15);
\draw (-20,-14) rectangle (-21,-15);
\node [below left] at (-20.5,-14.5) {\tiny(-21,-15)};
\node [below right] at (-20.5,-14.5) {\tiny(-20,-15)};
\node [above left] at (-20.5,-14.5) {\tiny(-21,-14)};
\node [above right] at (-20.5,-14.5) {\tiny(-20,-14)};
\draw plot[only marks, mark=*, mark options={fill=white}, mark size=.8ex] coordinates {(-20,-14) (-21,-14) (-20,-15) (-21,-15)};
\draw[dashed,red] (0,0) -- (-14,-20);
\draw[dashed,red] (0,0) -- (-15,-20);
\draw[dashed,red] (0,0) -- (-14,-21);
\draw[dashed,red] (0,0) -- (-15,-21);
\draw (-14,-20) rectangle (-15,-21);
\node [below left] at (-14.5,-20.5) {\tiny(-15,-21)};
\node [below right] at (-14.5,-20.5) {\tiny(-14,-21)};
\node [above left] at (-14.5,-20.5) {\tiny(-15,-20)};
\node [above right] at (-14.5,-20.5) {\tiny(-14,-20)};
\draw plot[only marks, mark=*, mark options={fill=white}, mark size=.8ex] coordinates {(-14,-20) (-15,-20) (-14,-21) (-15,-21)};
\end{tikzpicture}
\caption{Eight copies of the closest $2 \times 2$ hidden forest}
\label{fig:convention}
\end{center}
\end{figure}

\subsection{Quasiprime matrices and the \texttt{QP}-algorithm}\label{subsec:quasi}

In Proposition~\ref{prop:correspondence_betw_P_and_GcdP}, we begin with a prime matrix $P_n$ and observe that a $90^{\circ}$ counter-clockwise rotation of $P_n$ yields $\widetilde{\mathrm{Gcd}}_{P_n}$ which relates very closely to the $\gcd$-matrix $\mathrm{Gcd}_{P_n}$ of the corresponding $\Hforest{n}$ (in particular, recall that the $(i,j)$-entry of $\widetilde{\mathrm{Gcd}}_{P_n}$ divides the $(i,j)$-entry of $\mathrm{Gcd}_{P_n}$).  Now suppose instead that we start with an $\Hforest{n}$ and its associated $\gcd$-matrix, which we will denote $\mathrm{Gcd}_M$.  If we rotate this matrix $90^{\circ}$ clockwise, then we get some matrix $M$ that is not necessarily a prime matrix.  Furthermore, applying the \texttt{CRT}-algorithm on $M$ may not even be possible (see Example~\ref{exam:3by3}).  But from $M$, can we find a matrix $\widetilde{M}$ such that the $(i,j)$-entry of $\widetilde{M}$ divides the $(i,j)$-entry of $M$ and applying the \texttt{CRT}-algorithm on $\widetilde{M}$ gives the original $\Hforest{n}$ from which we started?  Based on much computational evidence, the answer appears to be yes.  The matrix $\widetilde{M}$ is what we call a quasiprime matrix $\QP$ to be defined in Definition \ref{def:quasiprime}, but a formal proof still awaits. For now, we proceed to give very substantial support that this conjecture holds for all $\Hforest{n}$ (see Question~\ref{question:every_hidden_forest_can_be_found_via_QP_matrix}).

To find $M$, we use the matrix equality in Proposition~\ref{prop:correspondence_betw_P_and_GcdP} and solve for $M$ as follows:
\begin{align}
\mathrm{Gcd}_M = (M \cdot AD_n)^T &\Longrightarrow (\mathrm{Gcd}_M)^T = M \cdot AD_n \notag \\
		&\Longrightarrow M = (\mathrm{Gcd}_M)^T \cdot AD_n, \label{eqn:QP_converse}
\end{align}
where Equation~(\ref{eqn:QP_converse}) follows since an anti-diagonal matrix with all ones in its nonzero entries is its own inverse.  Below we give an illustration of this process.
$$\mathrm{Gcd}_M = \left(
\begin{array}{cccccc}
g_{1,n} & g_{2,n} & \cdots & g_{i,n} & \cdots & g_{n,n} \\
\vdots & \vdots &  & \vdots & & \vdots\\
g_{1,j} & g_{2,j} & \cdots & g_{i,j} & \cdots & g_{n,j}\\
\vdots & \vdots &  & \vdots & & \vdots\\
g_{1,2} & g_{2,2} & \cdots & g_{i,2} & \cdots & g_{n,2}\\
g_{1,1} & g_{2,1} & \cdots & g_{i,1} & \cdots & g_{n,1}
\end{array}
\right)
\;\; \xmapsto[90^\circ \; \mathrm{right}]{\RotateRight} \;\;
M = \left(
\begin{array}{cccccc}
g_{1,1} & g_{1,2} & \cdots & g_{1,j} & \cdots & g_{1,n} \\
g_{2,1} & g_{2,2} & \cdots & g_{2,j} & \cdots & g_{2,n} \\
\vdots & \vdots &  & \vdots & & \vdots\\
g_{i,1} & g_{i,2} & \cdots & g_{i,j} & \cdots & g_{i,n}\\
\vdots & \vdots &  & \vdots & & \vdots\\
g_{n,1} & g_{n,2} & \cdots & g_{n,j} & \cdots & g_{n,n} 
\end{array}
\right).
$$

In the case of $n=2$, we see that $M$ is a prime matrix (see Example~\ref{exam:2by2}).  However in the case of $n=3$, the matrix $M$ can have repeated prime number entries and hence is not a prime matrix (see Example~\ref{exam:3by3}).  And in the case of $n\geq 4$, the matrix can have both repeated primes and composite number entries, and hence is not a prime matrix (see Example~\ref{exam:4by4}).  In these $n \geq 3$ cases, we construct a quasiprime version of $M$ which we denote $\QP$.  And an application of the \texttt{CRT}-algorithm on $\QP$ yields the $\Hforest{n}$ that has the original $\gcd$-matrix corresponding to $\Hforest{n}$.  Before we present an algorithm on how to produce $\QP$ from $M$, we give two motivating examples in the $n=2$ and $n=3$ cases.

\begin{example}[The closest $2 \times 2$]\label{exam:2by2}
By examining a small grid of points in Quadrant I of $\mathbb{Z}^2$, it is easy to notice that the closest $\Hforest{2}$ occurs at $x=14$ and $y=20$.  In the figure below, we give $H^2_{(14,20)}$ and to its right the $\gcd$-grid corresponding to the four nodes.
\begin{center}
\begin{tikzpicture}[scale=.75]
\draw (0,0)--(0,1)--(1,1)--(1,0)--(0,0);
\foreach \x in {0,1}
		\shade[ball color=blue] (\x,0) circle (.75ex);
\foreach \x in {0,1}
		\shade[ball color=blue] (\x,1) circle (.75ex);
\node [left] at (0,0) {\footnotesize (14,20)};
\node [left] at (0,1) {\footnotesize (14,21)};
\node [right] at (1,0) {\footnotesize (15,20)};
\node [right] at (1,1) {\footnotesize (15,21)};
\end{tikzpicture}
\hspace{1in}
\begin{tikzpicture}[scale=.75]
\draw (0,0)--(0,1)--(1,1)--(1,0)--(0,0);
\foreach \x in {0,1}
		\shade[ball color=blue] (\x,0) circle (.75ex);
\foreach \x in {0,1}
		\shade[ball color=blue] (\x,1) circle (.75ex);
\node [left] at (0,0) {\footnotesize 2};
\node [left] at (0,1) {\footnotesize 7};
\node [right] at (1,0) {\footnotesize 5};
\node [right] at (1,1) {\footnotesize 3};
\end{tikzpicture}
\end{center}
By Equation~(\ref{eqn:QP_converse}), we can retrieve a matrix $M$ from the $\gcd$-grid above as follows:
$$\mathrm{Gcd}_M = \left(
\begin{array}{cc}
7 & 3 \\
2 & 5 \end{array}
\right)
\;\; \xmapsto[90^\circ \; \mathrm{right}]{\RotateRight} \;\;
M=\left(
\begin{array}{cc}
2 & 7 \\
5 & 3 \end{array}
\right).
$$
Applying the \texttt{CRT}-algorithm to this matrix $M$, we get $x_0 = 13$ and $y_0 = 19$ as desired.  Hence at the distance of $d \approx 24.4131$ we have the closest hidden forest $H^2_{(14,20)}$.
\end{example}

\begin{remark}
Alternate permutations of the same primes in the gcd-grid, and consequently the matrix $M$, may produce different solutions under the \texttt{CRT}-algorithm. This means that for each unique set of $n^2$ primes in the matrix $P_n$, the \texttt{CRT}-algorithm can yield up to $n^2$ factorial (not necessarily distinct) $\Hforest{n}$. Compare the previous example with Example~\ref{exam:2by2_hidden_forest}. Both examples use the same set of primes $\{2,3,5,7\}$ but yield drastically different $\Hforest{2}$.
\end{remark}

\begin{example}[The closest $3 \times 3$]\label{exam:3by3}
At the distance of $d \approx 1825.91$ we find the closest hidden forest $H^3_{(1274,1308)}$.  Though others have cited $H^3_{(1274,1308)}$ as a hidden forest~\cite{Herzog, Weisstein}, none of these sources have asserted that it is the closest.  For the $n=3$ case, the problem of finding the closest hidden forest is computationally tractable via exhaustive means.  In fact, we have written Java code\footnote{See Appendix~\ref{appendix_A} for the Java code.} which exhaustively checked the square region with lower left endpoint $(0,0)$ and upper right endpoint $(1308,1308)$, finally confirming that this is the closest $3 \times 3$ hidden forest.  Below we give the hidden forest $H^3_{(1274,1308)}$ and its corresponding $\gcd$-grid:
\begin{center}
\begin{tikzpicture}[scale=.75]
\draw (0,0)--(0,2)--(2,2)--(2,0)--(0,0);
\draw (0,1)--(2,1);
\draw (1,0)--(1,2);
\foreach \x in {0,...,2}
		\shade[ball color=blue] (\x,0) circle (.75ex);
\foreach \x in {0,...,2}
		\shade[ball color=blue] (\x,1) circle (.75ex);
\foreach \x in {0,...,2}
		\shade[ball color=blue] (\x,2) circle (.75ex);
\node [left] at (0,0) {\tiny $(1274,1308)$};
\node [left] at (0,1) {\tiny $(1274,1309)$};
\node [left] at (0,2) {\tiny $(1274,1310)$};
\node [below] at (1,0) {\tiny $(1275,1308)$};
\node [below] at (1,1) {\tiny $(1275,1309)$};
\node [above] at (1,2) {\tiny $(1275,1310)$};
\node [right] at (2,0) {\tiny $(1276,1308)$};
\node [right] at (2,1) {\tiny $(1276,1309)$};
\node [right] at (2,2) {\tiny $(1276,1310)$};
\end{tikzpicture}
\hspace{.5in}
\begin{tikzpicture}[scale=.75]
\draw (0,0)--(0,2)--(2,2)--(2,0)--(0,0);
\draw (0,1)--(2,1);
\draw (1,0)--(1,2);
\foreach \x in {0,...,2}
		\shade[ball color=blue] (\x,0) circle (.75ex);
\foreach \x in {0,...,2}
		\shade[ball color=blue] (\x,1) circle (.75ex);
\foreach \x in {0,...,2}
		\shade[ball color=blue] (\x,2) circle (.75ex);
\node [left] at (0,0) {\footnotesize 2};
\node [left] at (0,1) {\footnotesize 7};
\node [left] at (0,2) {\footnotesize 2};
\node [below] at (1,0) {\footnotesize 3};
\node [below] at (1,1) {\footnotesize 17};
\node [above] at (1,2) {\footnotesize 5};
\node [right] at (2,0) {\footnotesize $2^2$};
\node [right] at (2,1) {\footnotesize 11};
\node [right] at (2,2) {\footnotesize 2};
\end{tikzpicture}
\end{center}
By Equation~(\ref{eqn:QP_converse}), we can retrieve a matrix $M$ from the $\gcd$-grid above as follows:
$$\mathrm{Gcd}_M = \left(
\begin{array}{ccc}
2 & 5 & 2 \\
7 & 17 & 11 \\
2 & 3 & 2^2 \end{array}
\right)
\;\; \xmapsto[90^\circ \; \mathrm{right}]{\RotateRight} \;\;
M=\left(
\begin{array}{ccc}
2 & 7 & 2 \\
3 & 17 & 5 \\
2^2 & 11 & 2 \end{array}
\right).
$$
Incidentally, the $M$ given here is very similar to the one given in 1971 by Herzog and Stewart~\cite{Herzog}, but neither neither of these matrices can possibly produce the correct $\Hforest{3}$ because the Chinese Remainder Theorem simply cannot work on such matrices.  For example, since the products of row 1 and row 3 of $M$ each have a factor of 4, then any solution $x_0$ to the three row equations would also have to satisfy $x+1 \equiv 0 \pmod{4}$ and $x+3 \equiv 0 \pmod{4}$, but the existence of such an $x_0$ is absurd.  However, this problem is resolved by introducing the concept of a quasiprime matrix.
\end{example}

\begin{definition}\label{def:quasiprime}
Given a matrix $M$ arising from a $\mathrm{Gcd}_M$ via Equation~(\ref{eqn:QP_converse}), we produce a \textit{quasiprime} matrix $\QP$ defined by the \texttt{QP}\textit{-algorithm} given below.

\bigskip

\boitegrise{
\noindent\textbf{The \texttt{QP}-algorithm to construct a quasiprime matrix $\QP$:}
\begin{enumerate}
\item Construct matrix $M$ arising from a $\mathrm{Gcd}_M$ via Equation~(\ref{eqn:QP_converse}).
\item Let $\{p_i\}_{i=1}^s$ be the union of the sets of all primes appearing in the prime factorizations of each entry of $M$.
\item For a fixed $p_i$ with $1 \leq i \leq s$, locate the entry in $M$ which contains $p_i^{k}$ for $k \geq 1$ such that $k$ is largest.  If there is more than one entry which contains $p_i^{k}$, then choose exactly one.
\item Place the selected $p_i^k$ in $\QP$ in the same location where it appears in $M$.  Place the value 1 in $\QP$ in every location where a $p_i^j$ appears in $M$ for each $j \leq k$.
\item Repeat the previous two steps for each $p_i$ with $1 \leq i \leq s$.
\end{enumerate}\vspace{-.2in}}{0.75\textwidth}

\end{definition}

\begin{example}[The closest $3 \times 3$ via a quasiprime matrix]
From the matrix $M$ in Example~\ref{exam:3by3}, we can produce the quasiprime matrix as follows using the \texttt{QP}-algorithm:
$$M=\left(
\begin{array}{ccc}
2 & 7 & 2 \\
3 & 17 & 5 \\
2^2 & 11 & 2 \end{array}
\right)
\;\; \xmapsto[\mathrm{algorithm}]{\mathrm{QP}} \;\;
\QP=\left(
\begin{array}{ccc}
1 & 7 & 1 \\
3 & 17 & 5 \\
2^2 & 11 & 1 \end{array}
\right).
$$
By use of the \texttt{CRT}-algorithm on $\QP$, we solve the following system of linear congruences
$$\begin{array}{lcl}
x+1 \equiv 0 \pmod{7} & \hspace{.2in} & y+1 \equiv 0 \pmod{2^2 \cdot 3}\\
x+2 \equiv 0 \pmod{3 \cdot 5 \cdot 17} &  \hspace{.2in} & y+2 \equiv 0 \pmod{7 \cdot 11 \cdot 17}\\
x+3 \equiv 0 \pmod{2^2 \cdot 11} &  \hspace{.2in} & y+3 \equiv 0 \pmod{5}
\end{array}$$
which has solutions $x_0 = 1273$ and $y_0 = 1307$.  Hence the $\QP$ yields the closest $3 \times 3$ hidden forest $H^3_{(1274,1308)}$.
\end{example}

\subsection{Computer-heavy approach: Strings of strongly composite integers}\label{subsec:strings}

Another technique we implement that proves very powerful in finding hidden forests involves using strings of consecutive integers each with several prime factors.  In Section~\ref{sec:the_great_unknown}, we find that combining the technique below with a clever computational use of quasiprime matrices yields the closest known $n \times n$ hidden forests for $n \geq 4$.

In 1990, Schumer proved that there exist strings of $n$ consecutive integers each divisible by at least $k$ distinct primes, which he calls strings of strongly composite integers~\cite{Schumer}.  The proof uses the Chinese Remainder Theorem, and hence like Theorem~\ref{thm:2by2_case} it produces very large numbers for $n \geq 3$.  However, there is an efficient way to find the smallest set of $n$ consecutive integers each with at least $k$ prime factors each (ignoring multiplicity).  We can then use these values as our $x$-values in our hunt for closer hidden forests for $n \geq 4$.

The following \texttt{Mathematica} code easily produces the very first number $n$ in a sequence of four consecutive integers each with four prime factors (ignoring multiplicity):
  \begin{lstlisting}[frame=single,language=Mathematica,caption={Nayuki Minase's solution to Project Euler Problem 47}, label={mathematica_code}]
    Has4PrimeFactors[n_] := Length[FactorInteger[n]] == 4
		i = 2;
		While[! (Has4PrimeFactors[i] && Has4PrimeFactors[i + 1] && 
    Has4PrimeFactors[i + 2] && Has4PrimeFactors[i + 3]), i++]
		i
  \end{lstlisting}
The value that it yields is 134,043.  This number and the next three consecutive integers have the following prime factorizations:
\begin{align*}
\numprint{134 043} &= 3 \cdot 7 \cdot 13 \cdot 491\\
\numprint{134 044} &= 2^2 \cdot 23 \cdot 31 \cdot 47\\
\numprint{134 045} &= 5 \cdot 17 \cdot 19 \cdot 83\\
\numprint{134 046} &= 2 \cdot 3^2 \cdot 11 \cdot 677.
\end{align*}

\begin{example}[The closest $4 \times 4$ to date with the smallest $x$-value]\label{exam:4by4}
Using the four values \numprint{134 043} to \numprint{134 046} as the $x$-values of $4 \times 4$ hidden forest we seek, we exhaustively searched for the very first set of four consecutive integers such that all four values share at least one prime factor with each of the four values \numprint{134 043} to \numprint{134 046}.  After running for only two minutes\footnote{The run times for the Java code are based on the code running on the Blugold Supercomputing Cluster of UWEC, while the run times for the \texttt{Mathematica} code are based on the code running on a standard office computer.}, the Java program which we wrote outputs the value \numprint{184 785 885}.  This number and the next three consecutive integers have the following prime factorizations:
\begin{align*}
\numprint{184 785 885} &= 3^2 \cdot 5 \cdot 31^2 \cdot 4273\\
\numprint{184 785 886} &= 2 \cdot 17 \cdot 491 \cdot \numprint{11 069}\\
\numprint{184 785 887} &= 11 \cdot 13 \cdot 19 \cdot 23 \cdot 2957\\
\numprint{184 785 888} &= 2^5 \cdot 3 \cdot 7 \cdot 83 \cdot 3313.
\end{align*}
Using these four numbers as the $y$-values of our $\Hforest{4}$ and the four values 134,043 to 134,046 as the $x$-values, we get a hidden forest $H^4_{(134043,184785885)}$ with the following $\gcd$-grid:
\begin{center}
\begin{tikzpicture}[scale=.75]
\draw (0,0)--(0,3)--(3,3)--(3,0)--(0,0);
\draw (0,1)--(3,1);
\draw (0,2)--(3,2);
\draw (1,0)--(1,3);
\draw (2,0)--(2,3);
\foreach \x in {0,...,3}
		\shade[ball color=blue] (\x,0) circle (.75ex);
\foreach \x in {0,...,3}
		\shade[ball color=blue] (\x,1) circle (.75ex);
\foreach \x in {0,...,3}
		\shade[ball color=blue] (\x,2) circle (.75ex);
\foreach \x in {0,...,3}
		\shade[ball color=blue] (\x,3) circle (.75ex);
\node [left] at (0,0) {\footnotesize 3};
\node [left] at (0,1) {\footnotesize 491};
\node [left] at (0,2) {\footnotesize 13};
\node [left] at (0,3) {\footnotesize $3 \! \cdot \! 7$};
\node [below] at (1,0) {\footnotesize 31};
\node [below] at (1,1) {\footnotesize 2};
\node [below] at (1,2) {\footnotesize 23};
\node [above] at (1,3) {\footnotesize $2^2$};
\node [below] at (2,0) {\footnotesize 5};
\node [below] at (2,1) {\footnotesize 17};
\node [below] at (2,2) {\footnotesize 19};
\node [above] at (2,3) {\footnotesize 83};
\node [right] at (3,0) {\footnotesize $3^2$};
\node [right] at (3,1) {\footnotesize 2};
\node [right] at (3,2) {\footnotesize 11};
\node [right] at (3,3) {\footnotesize $2 \! \cdot \! 3$};
\end{tikzpicture}
\end{center}
By Equation~(\ref{eqn:QP_converse}), we retrieve a matrix $M$ from the $\gcd$-grid above as follows:
$$\mathrm{Gcd}_M = \left(
\begin{array}{cccc}
3 \! \cdot \! 7 & 2^2 & 83 & 2 \! \cdot \! 3 \\
13 & 23 & 19 & 11 \\
491 & 2 & 17 & 2 \\
3 & 31 & 5 & 3^2 \end{array}
\right)
\;\; \xmapsto[90^\circ \; \mathrm{right}]{\RotateRight} \;\;
M=\left(
\begin{array}{cccc}
3 & 491 & 13 & 3 \! \cdot \! 7 \\
31 & 2 & 23 & 2^2 \\
5 & 17 & 19 & 83 \\
3^2 & 2 & 11 & 2 \! \cdot \! 3 \end{array}
\right).
$$
Applying the \texttt{QP}-algorithm to $M$, we get the following quasiprime matrix
$$
\QP=\left(
\begin{array}{cccc}
1 & 491 & 13 & 7 \\
31 & 1 & 23 & 2^2 \\
5 & 17 & 19 & 83 \\
3^2 & 1 & 11 & 1 \end{array}
\right).
$$
Then applying the \texttt{CRT}-algorithm to $\QP$ does indeed yield the forest $H^4_{(134043,184785885)}$ as desired.
\end{example}

\begin{remark}\label{rem:4x4_far_forest}
The hidden forest $H^4_{(134043,184785885)}$ is distance $d \approx 1.84786 \times 10^{8}$ from the origin.  In comparison recall Example \ref{exam:4by4_hidden_forest}, where we used the only known method to date in the literature (that is, the prime matrix $P_4$ and Theorem~\ref{thm:2by2_case}). Using that traditional method we found the hidden forest $H^4_{(x_1,y_1)}$ with
\begin{align*}
x_1 &= \numprint{2 847 617 195 518 191 810}\\
y_1 &= \numprint{1 160 906 121 308 397 398}
\end{align*}
at a distance $d \approx 3.07516 \times 10^{18}$ which is $1.66418 \times 10^{10}$ times farther than the hidden forest which we found in Example~\ref{exam:4by4}!  The matrix $P_4$ and its associated $\gcd$-grid of our closer hidden forest $H^4_{(134043,184785885)}$ is as follows:
$$
P_4 = \left(
\begin{array}{cccc}
2 & 3 & 5 & 7 \\
11 & 13 & 17 & 19 \\
23 & 29 & 31 & 37 \\
41 & 43 & 47 & 53 \end{array}
\right)
\;\; \xmapsto[\mathrm{\texttt{CRT}-algorithm}]{\mathrm{Theorem~\ref{thm:2by2_case}}} \;\;
\begin{tikzpicture}[baseline=(current bounding box.center), scale=.75]
\draw (0,0)--(0,3)--(3,3)--(3,0)--(0,0);
\draw (0,1)--(3,1);
\draw (0,2)--(3,2);
\draw (1,0)--(1,3);
\draw (2,0)--(2,3);
\foreach \x in {0,...,3}
		\shade[ball color=blue] (\x,0) circle (.75ex);
\foreach \x in {0,...,3}
		\shade[ball color=blue] (\x,1) circle (.75ex);
\foreach \x in {0,...,3}
		\shade[ball color=blue] (\x,2) circle (.75ex);
\foreach \x in {0,...,3}
		\shade[ball color=blue] (\x,3) circle (.75ex);
\node [left] at (0,0) {\footnotesize 2};
\node [left] at (0,1) {\footnotesize 3};
\node [left] at (0,2) {\footnotesize $2 \! \cdot \! 5$};
\node [left] at (0,3) {\footnotesize 7};
\node [below] at (1,0) {\footnotesize 11};
\node [below] at (1,1) {\footnotesize 13};
\node [below] at (1,2) {\footnotesize 17};
\node [above] at (1,3) {\footnotesize 19};
\node [below] at (2,0) {\footnotesize $2 \! \cdot \! 23$};
\node [below] at (2,1) {\footnotesize 29};
\node [below] at (2,2) {\footnotesize $2^2 \! \cdot \! 31$};
\node [above] at (2,3) {\footnotesize 37};
\node [right] at (3,0) {\footnotesize 41};
\node [right] at (3,1) {\footnotesize $3 \! \cdot \! 43$};
\node [right] at (3,2) {\footnotesize 47};
\node [right] at (3,3) {\footnotesize 53};
\end{tikzpicture}
$$
Hence the $\gcd$-matrix of this $4 \times 4$ hidden forest is
$$\mathrm{Gcd}_M = \left(
\begin{array}{cccc}
7 & 19 & 37 & 53 \\
2 \! \cdot \! 5 & 17 & 2^2 \! \cdot \! 31 & 47 \\
3 & 13 & 29 & 3 \! \cdot \! 43 \\
2 & 11 & 2 \! \cdot \! 23 & 41 \end{array}
\right).
$$
\end{remark}

\begin{remark}
Other researchers have found $4 \times 4$ hidden forests of distances relatively close to the one shown in Remark~\ref{rem:4x4_far_forest}.  In 2002 Pighizzini and Shallit addressed the issue of finding the closest $n \times n$ hidden forests~\cite{Pigh02}.  For a positive integer $n$, they define a function $S(n)$, which is the least positive integer $r$ such that there exists $m \in \{0,1,\ldots,r\}$ with $\gcd(r-i,m-j)>1$ for $0 \leq i,j < n$. This is equivalent to finding the closest $n \times n$ hidden forest. They were only successful in finding this value for $n=1, 2, 3$, but for $n=4$ they were able to give bounds $450000 < S(4) \leq 172379781$ by finding a hidden forest $\Hforest{4}$ with $x = \numprint{172 379 778}$ and $y = \numprint{153 132 342}$. An even closer $4 \times 4$ hidden forest was later revealed in 2013 in a book by Baake and Grimm~\cite[pg.422]{Baake2013}. The forest they find has bottom left corner $x = \numprint{13 458 288}$ and $y = \numprint{13 449 225}$ however no proof or justification of how this was found is given. Moreover, they give no assertion regarding whether this is the closest known $4 \times 4$ hidden forest. The following table gives the distances of the three closest known $4 \times 4$ hidden forests in the literature.
\end{remark}

\begin{center}
\begin{tabular}{|c||c|c|c|}
\hline
 & Year & Distance of $\Hforest{4}$ & Proof/method given? \\ \hline\hline
Pighizzini and Shallit 		& 2002 &  $2.30574 \times 10^{8}$    & No\\ \hline
Baake and Grimm  & 2013 & $1.90265 \times 10^{7}$     & No\\ \hline
Goodrich, Mbirika, and Nielsen 	& 2014\,\footnotemark & $1.84786 \times 10^{8}$     & Yes\\ \hline     
\end{tabular}
\footnotetext{We discovered this forest in 2014; however, it is in this 2020 paper in which we give its existence and proof.}
\end{center}

\subsection{Computer-free approach: Minimum prime factors in an optimal gcd-matrix}\label{subsec:min_prime_factors}

The concept of an ``optimal'' $\gcd$-matrix for a hidden $n \times n$ forest $\Hforest{n}$ depends on $n$ and is based on the minimal number of prime factors required in the $\gcd$-grid of $\Hforest{n}$.  We find that minimizing the number of primes used in the $\gcd$-matrix while simultaneously maximizing the number of locations in the $\gcd$-grid where a prime can be used again leads to a closer $\Hforest{n}$ than the traditional method given in Section~\ref{sec:the_CRT_algorithm}.

Observe that the $\gcd$-matrix of the $H_{x,y}^{(4)}$ in Remark~\ref{rem:4x4_far_forest} is hardly optimal in the sense that if the corner entries were all multiples of 3, then we immediately get the four corners ``hidden for free'', as in the forest in Example~\ref{exam:4by4}---that is, the values $x_1, x_4, y_1$ and $y_4$ would all be divisible by 3 and hence none of the four points $(x_1,y_1)$, $(x_4,y_1)$, $(x_1,y_4)$, or $(x_4,y_4)$ would be visible.  An optimal situation is to have one corner, for example, the bottom-left coordinate $(x_1,y_1)$ to be divisible by both 2 and 3.  Then we would have a forest where the $\gcd$ of the following 16 coordinates are divisible by 2, 3, and nine other primes $p_1, \ldots, p_9$ as in Figure~\ref{fig:optimal_4_by4_grid}.
\begin{figure}[H]
\begin{center}
$$
\begin{tikzpicture}[baseline=(current bounding box.center), scale=.75]
\draw (0,0)--(0,3)--(3,3)--(3,0)--(0,0);
\draw (0,1)--(3,1);
\draw (0,2)--(3,2);
\draw (1,0)--(1,3);
\draw (2,0)--(2,3);
\foreach \x in {0,...,3}
		\shade[ball color=blue] (\x,0) circle (.75ex);
\foreach \x in {0,...,3}
		\shade[ball color=blue] (\x,1) circle (.75ex);
\foreach \x in {0,...,3}
		\shade[ball color=blue] (\x,2) circle (.75ex);
\foreach \x in {0,...,3}
		\shade[ball color=blue] (\x,3) circle (.75ex);
\node [left] at (0,0) {\footnotesize $2\times 3$};
\node [left] at (0,1) {\footnotesize $p_3$};
\node [left] at (0,2) {\footnotesize 2};
\node [left] at (0,3) {\footnotesize 3};
\node [below] at (1,0) {\footnotesize $p_1$};
\node [below] at (1,1) {\footnotesize $p_2$};
\node [below] at (1,2) {\footnotesize $p_5$};
\node [above] at (1,3) {\footnotesize $p_9$};
\node [below] at (2,0) {\footnotesize 2};
\node [below] at (2,1) {\footnotesize $p_4$};
\node [below] at (2,2) {\footnotesize 2};
\node [above] at (2,3) {\footnotesize $p_8$};
\node [right] at (3,0) {\footnotesize 3};
\node [right] at (3,1) {\footnotesize $p_6$};
\node [right] at (3,2) {\footnotesize $p_7$};
\node [right] at (3,3) {\footnotesize 3};
\end{tikzpicture}
$$
\end{center}
\vspace{-.25in}
\caption{}\label{fig:optimal_4_by4_grid}
\end{figure}

This leads one to consider a different type of $\gcd$-matrix that does not give the exact $\gcd$ $g_{i,j}$ (recall Figure~\ref{fig:generic_gcd_grid}) for each coordinate $(x_i, x_j)$ (recall Figure~\ref{fig:generic_hidden_forest}) of $\Hforest{n}$.  But on the other hand, this new matrix would simply give the smallest prime divisor of the $\gcd$ for each coordinate.  We make this more precise in Definition~\ref{defn:optimal_gcd_matrix}.  But first we need to recall the following number-theoretic function.

\begin{definition}
The \textit{prime counting function} $\pi : \mathbb{R} \rightarrow \mathbb{N}$ counts the number of primes less than or equal to a given real number.
\end{definition}

\begin{definition}\label{defn:optimal_gcd_matrix}
Construct an \textit{optimal gcd-matrix} as follows.  Let one of the four corner entries of the $n \times n$ matrix contain the product of the first $k_n := \pi(n)$ primes (where $\pi$ is the prime counting function).  Without loss of generality, choose the bottom-left corner for this value.  Denote these first $k_n$ primes as $q_1, q_2, \ldots, q_{k_n}$.  For each $q_i$ with $1 \leq i \leq k_n$, any entry in the matrix that is a multiple of $q_i$ rows to the right of the bottom-left corner and/or a multiple of $q_i$ columns above the bottom-left corner must be filled with the value $q_i$.  If more than one prime fits this criteria for a specific matrix entry, then simply multiply the primes in that entry together.  In the remaining unfilled entries, place one prime in each entry from the set of the next smallest primes larger than the prime $q_{k_n}$.  Denote this set of primes by $\{p_1, p_2, \ldots\}$.  We denote this optimal $\gcd$-matrix by the symbol $\optGcd$.
\end{definition}

\begin{example}
An optimal $4 \times 4$ $\gcd$-matrix is
$$
\optGcd = \left(
\begin{array}{cccc}
3 & p_9 & p_8 & 3 \\
2 & p_5 & 2 & p_7 \\
p_3 & p_2 & p_4 & p_6 \\
2 \times 3 & p_1 & 2 & 3\end{array}
\right)
$$
where the entries $p_1, \ldots, p_9$ are the nine smallest prime numbers other than 2 or 3.  Observe that the locations of the primes 2 and 3 correspond exactly to their location in the $\gcd$-grid in Figure~\ref{fig:optimal_4_by4_grid}. The manner in which the $p_1$ through $p_9$ are distributed in this particular matrix is the $n=4$ case that arises from the grid in Figure~\ref{fig:optimal_gcd_grid}.
\end{example}

The grid in Figure~\ref{fig:optimal_gcd_grid} shows us the minimum number of primes and their relative locations in a candidate for an optimal $\gcd$-matrix for an $n \times n$ hidden forest.  In this grid, we choose the bottom-left corner (denoted with the symbol $\bullet$) to contain the product of powers of the first $k_n$ primes where $k_n$ is the value given in Definition~\ref{defn:optimal_gcd_matrix}.  In the far-left shaded column, each entry refers to the size $n$ of the corresponding $n \times n$ grid.  In the bottom shaded row, in each box we give the number of additional primes that are needed to go from an $n \times n$ grid to an $(n+1) \times (n+1)$ grid.  For example, for $n=5$, we need a minimum $k_5 + 3 + 2 + 4 + 4 = 15$ distinct primes in the optimal $\gcd$-matrix for an $H^{(5)}_{x,y}$.  Indeed in Section~\ref{sec:the_great_unknown}, we see that this minimum is achieved.

\begin{figure}[H]
\centering
		\begin{minipage}{.5\textwidth}
			\centering
			\includegraphics[width=1\linewidth, height=0.4\textheight]{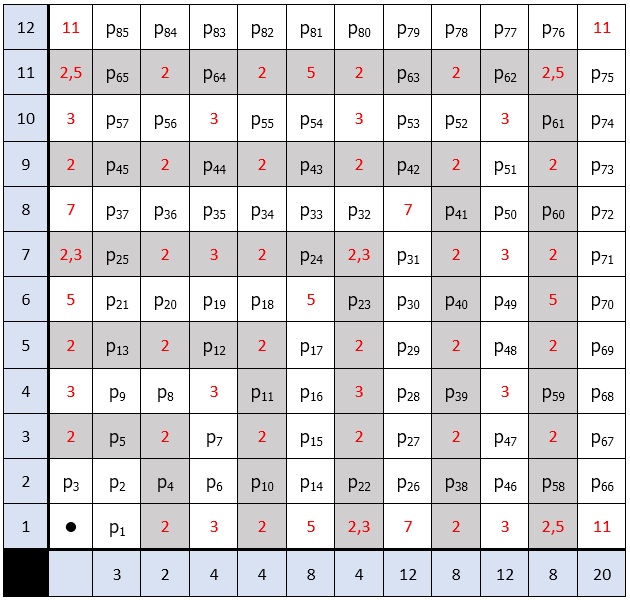}
		\end{minipage}
\hspace{.1in}
		\begin{minipage}{.3\textwidth}
			\centering
			\begin{tabular}{c|c|c}
			$n$ & $k_n$ & $\bullet$-value                \\ \hline
			2   & 1     & 2                              \\
			3   & 1     & 2                              \\
			4   & 2     & $2 \times 3$                   \\
			5   & 2     & $2 \times 3$                   \\
			6   & 3     & $2 \times 3 \times 5$          \\
			7   & 3     & $2 \times 3 \times 5$          \\
			8   & 4     & $2 \times 3 \times 5 \times 7$ \\
			9   & 4     & $2 \times 3 \times 5 \times 7$ \\
			10  & 4     & $2 \times 3 \times 5 \times 7$ \\
			11  & 4     & $2 \times 3 \times 5 \times 7$ \\
			12  & 5     & $2 \times 3 \times 5 \times 7 \times 12$
			\end{tabular}
		\end{minipage}
\caption{An optimal gcd-grid with bottom left corner having the product of $k_n$ primes}\label{fig:optimal_gcd_grid}
\end{figure}

\begin{example}\label{exam:computer_free_4_by_4}
If $n=4$ then $k_4 = 2$, and hence we place $2\times 3$ in a corner location.  This is because the $4 \times 4$ portion of the grid in Figure~\ref{fig:optimal_gcd_grid} says that we need a minimum of 9 primes, not counting the primes 2 and 3 which are placed in the locations where they appear in the grid.  Hence an \textit{optimal gcd-matrix} might be as follows:
$$\optGcd = \left(
\begin{array}{cccc}
3 & \framebox[1.1\width]{29} & \framebox[1.1\width]{31} & 3 \\
2 & \framebox[1.1\width]{19} & 2 & \framebox[1.1\width]{23} \\
\framebox[2.2\width]{7} & \framebox[1.1\width]{11} & \framebox[1.1\width]{13} & \framebox[1.1\width]{17} \\
2 \times 3 & \framebox[2.2\width]{5} & 2 & 3 \end{array}
\right).
$$
In the boxed entries in the matrix above, we place the 9 smallest primes larger than 3 where the grid in Figure~\ref{fig:optimal_gcd_grid} places $p_1, \ldots, p_9$.

Observe that the forest $H_{(134043,184785885)}^4$ found in Example~\ref{exam:4by4} is the closest known $4 \times 4$ forest and is attained by cleverly using the method of \textit{strings of strongly composite integers}. That is, using computer computation to find the smallest four consecutive values $x_1, \ldots, x_4$ which each have at least 4 primes factors each, and then using computer computation again to compute the next set of four values $y_1, \ldots, y_4$ each of which is not relatively prime to all four $x$-values.  However, the $\QP$ associated to this closest forest is not optimal in the sense that it uses 10 primes (not including 2 and 3), whereas an optimal $\QP$ uses at most 9 primes (not including 2 and 3).

Note that much computer assistance was required to generate $H_{(134043,184785885)}^4$, however no computer assistance whatsoever is required to create the optimal $\gcd$-matrix, $\optGcd$.  From the matrix $\optGcd$ in this example, we can produce the quasiprime matrix as follows using the \texttt{QP}-algorithm:
$$\optGcd = \left(
\begin{array}{cccc}
3 & 29 & 31 & 3 \\
2 & 19 & 2 & 23 \\
7 & 11 & 13 & 17 \\
2 \times 3 & 5 & 2 & 3 \end{array}
\right)
\;\; \xmapsto[\mathrm{algorithm}]{\mathrm{QP}} \;\;
\QP = \left(
\begin{array}{cccc}
1 & 29 & 31 & 1 \\
1 & 19 & 1 & 23 \\
7 & 11 & 13 & 17 \\
6 & 5 & 1 & 1 \end{array}
\right).
$$
Applying the \texttt{CRT}-algorithm on $\QP$, we then get the forest $\Hforest{4}$ with $x = \numprint{153 630 616 137}$ and $y = \numprint{116 380 988 514}$ and the following prime factorizations of the 16 coordinates $(x_i, y_j)$ for all $1 \leq i,j \leq 4$:
{\footnotesize
\begin{align*}
x_1 &= \numprint{153 630 616 137} = 3 \cdot 29 \cdot 31 \cdot 229 \cdot \numprint{248 749} &\hspace{.2in} y_1 &= \numprint{116 380 988 514} = 2 \cdot 3^3 \cdot 7 \cdot 37 \cdot \numprint{8 321 249}\\
x_2 &= \numprint{153 630 616 138} = 2 \cdot 19 \cdot 23 \cdot 1723 \cdot \numprint{102 019} & \hspace{.2in} y_2 &= \numprint{116 380 988 515} = 5 \cdot 11 \cdot 19 \cdot 29 \cdot 47 \cdot 101 \cdot 809\\
x_3 &= \numprint{153 630 616 139} = 7 \cdot 11 \cdot 13 \cdot 17 \cdot \numprint{9 028 067} &\hspace{.2in} y_3 &= \numprint{116 380 988 516} = 2^2 \cdot 13 \cdot 31 \cdot \numprint{72 196 643}.\\
x_4 &= \numprint{153 630 616 140} = 2^2 \cdot 3^3 \cdot 5 \cdot 103 \cdot \numprint{2 762 147} &\hspace{.2in} y_4 &= \numprint{116 380 988 517} = 3 \cdot 17 \cdot 23 \cdot 883 \cdot \numprint{112 363}.
\end{align*}
}

Below we summarize the distances of the $4 \times 4$ hidden forests found by the traditional method versus the two new methods given in this paper.

\begin{center}
\begin{tabular}{|c||c|c|}
\hline
Method & Distance & Location in paper\\ \hline\hline
Traditional approach: \texttt{CRT}-algorithm 		&  $3.07516 \times 10^{18}$    & Example~\ref{exam:4by4_hidden_forest}\\ \hline
Computer-heavy approach with \texttt{QP}-algorithm  & $1.84786 \times 10^{8}$     & Example~\ref{exam:4by4}\\ \hline
Computer-free approach with \texttt{QP}-algorithm 	& $1.92735 \times 10^{11}$     & Example~\ref{exam:computer_free_4_by_4}\\ \hline              
\end{tabular}
\end{center}
So we can easily see that the two new methods produce substantially closer hidden forests than the traditional methods.  However, we find that merging the method of strings of composite integers with the method of the optimal matrix is an even better idea.  And that is precisely what we do in the $5 \times 5$ case in the following section.
\end{example}


\section{An application: the closest known \texorpdfstring{$5 \times 5$}{5 x 5} hidden forest}\label{sec:the_great_unknown}

We employ a combination of the techniques from both strings of strongly composite integers and an optimal quasiprime matrix to find the closest known $5 \times 5$ hidden forest to date.  We first calculate a length~5 analogue of the Project Euler Problem 47 (recall the footnote given in Remark~\ref{rem:Project_Euler}).  By slightly altering Minase's solution (see Listing~\ref{mathematica_code}), we find the smallest set of five consecutive integers each with at least five prime factors.  \texttt{Mathematica} completed this computation in 36 minutes.  These five integers and their prime factorizations are
\begin{align*}
x_1 = \numprint{129 963 314} &= 2 \cdot 13 \cdot 37 \cdot 53 \cdot 2549\\
x_2 = \numprint{129 963 315} &= 3 \cdot 5 \cdot 31 \cdot 269 \cdot 1039\\
x_3 = \numprint{129 963 316} &= 2^2 \cdot 7 \cdot 97 \cdot 109 \cdot 439\\
x_4 = \numprint{129 963 317} &= 11^2 \cdot 17 \cdot 23 \cdot 41 \cdot 67\\
x_5 = \numprint{129 963 318} &= 2 \cdot 3 \cdot 89 \cdot 199 \cdot 1223.
\end{align*}
In Example~\ref{exam:4by4}, it took the Java code only 2 minutes to find the smallest four consecutive values which are each not relatively prime to the four values \numprint{134 043} through \numprint{134 046}.  However in this $n=5$ case, it is not as simple.  After the Java code ran continuously for four days
\footnote{Recall the footnote regarding computation times in Example~\ref{exam:4by4}.}
, it had checked up to the $y$-value of 500 billion and still did not find an $\Hforest{5}$ with the $x$-values $x_1, \ldots, x_5$ given earlier.  So we approached this problem from a more theoretical perspective instead.

Consider the list of five consecutive integers $x_1, \ldots, x_5$.  Observe that $x_1$, $x_3$, and $x_5$ are divisible by 2 and that $x_2$ and $x_5$ are divisible by 3.  Hence a hidden $5 \times 5$ forest bearing these $x$-values would be ``optimal'' if the corresponding five $y$-values (which we denote $y_1, \ldots, y_5$) have the property that $y_1$, $y_3$, and $y_5$ are divisible by 2 and that $y_2$ and $y_5$ are divisible by 3.  The benefit of this optimal situation is that 12 of the 25 coordinates will automatically have $\gcd(x_i,y_j) > 1$ and hence these 12 points are hidden.  In the matrices below, we represent each of these 12 points with the symbol $\bullet$ in the $\gcd$-matrix $\mathrm{Gcd}_M$ on the left, and to its right we give the $90^{\circ}$ clockwise rotation matrix $M$ from which we construct a quasiprime matrix.
$$\mathrm{Gcd}_M=
\begin{blockarray}{rlccccc}
		\begin{block}{rl(ccccc)}
			y_5 & \rightarrow & \bullet & \bullet & \bullet & d_5 & \bullet \\
			y_4 & \rightarrow & a_2 & b_3 & c_2 & d_4 & e_1 \\
			y_3 & \rightarrow & \bullet & b_2 & \bullet & d_3 & \bullet \\
			y_2 & \rightarrow & a_1 & \bullet & c_1 & d_2 & \bullet \\
			y_1 & \rightarrow & \bullet & b_1 & \bullet & d_1 & \bullet \\
		\end{block}
	& & \uparrow & \uparrow & \uparrow &  \uparrow & \uparrow \\	
	& & x_1 & x_2 & x_3 & x_4 & x_5 \\
\end{blockarray}
\; \xmapsto[90^\circ \; \mathrm{right}]{\RotateRight} \;
M=
\begin{blockarray}{rlccccc}
	& & y_1 & y_2 & y_3 & y_4 & y_5 \\
	& & \downarrow & \downarrow & \downarrow &  \downarrow & \downarrow \\
		\begin{block}{rl(ccccc)}
			x_1 & \rightarrow & \bullet & a_1 & \bullet & a_2 & \bullet \\
			x_2 & \rightarrow & b_1 & \bullet & b_2 & b_3 & \bullet \\
			x_3 & \rightarrow & \bullet & c_1 & \bullet & c_2 & \bullet \\
			x_4 & \rightarrow & d_1 & d_2 & d_3 & d_4 & d_5 \\
			x_5 & \rightarrow & \bullet & \bullet & \bullet & e_1 & \bullet \\
		\end{block}
\end{blockarray}
$$
Since we know that $x_5$ and $y_5$ are both divisible by 2 and 3 in this optimal case, we place a 6 in this entry, and the $\QP$ matrix has the following abstract form
\begin{equation}\label{matrix:5x5}
\QP = \left(
\begin{array}{ccccc}
1 & a_1 & 1 & a_2 & 1 \\
b_1 & 1 & b_2 & b_3 & 1 \\
1 & c_1 & 1 & c_2 & 1 \\
d_1 & d_2 & d_3 & d_4 & d_5 \\
1 & 1 & 1 & e_1 & 6 \end{array}
\right)
\end{equation}
where
$$
\begin{array}{rcrl}
x_1 = 2 \cdot 13 \cdot 37 \cdot 53 \cdot 2549 & \Longrightarrow & a_1,a_2 &\in \{ 13, 37, 53, 2549\}, \\
x_2 = 3 \cdot 5 \cdot 31 \cdot 269 \cdot 1039 & \Longrightarrow & b_1,b_2, b_3 &\in \{ 5, 31, 269, 1039\}, \\
x_3 = 2^2 \cdot 7 \cdot 97 \cdot 109 \cdot 439 & \Longrightarrow & c_1,c_2 &\in \{ 7, 97, 109, 439\}, \\
x_4 = 11^2 \cdot 17 \cdot 23 \cdot 41 \cdot 67 & \Longrightarrow & d_1,d_2, d_3, d_4, d_5 &\in \{ 11, 17, 23, 41, 67\}, \mbox{ and} \\
x_5 = 2 \cdot 3 \cdot 89 \cdot 199 \cdot 1223 & \Longrightarrow & e_1 &\in \{ 89, 199, 1223\}.
\end{array}
$$

\begin{observation}\label{lem:5_by_5_CRT_on_QP_matrix}
Consider the $\QP$ matrix in $(\ref{matrix:5x5})$. Then the following hold.
\begin{enumerate}[(1)]
\item There are \numprint{1 244 160} distinct ways to produce a quasiprime matrix $\QP$.
\item Applying the \texttt{CRT}-algorithm to any of the $\QP$ yields the same solution values $x_1, \ldots, x_5$ as the $x$-values of the hidden forest $H_{(x_1,y_1)}^5$. In particular, this unique $x_1$ value is \numprint{129 963 314}.
\item The $y$-value solutions have the property that $y_1, y_3, y_5 \in 2\mathbb{Z}$ and $y_2, y_5 \in 3\mathbb{Z}$.
\end{enumerate}
\end{observation}

\noindent \textbf{Proof of (1):} There are $P(4,2) = \frac{4!}{(4-2)!} = 12$ possible 2-permutations of a 4-element set.  So since $a_1$ and $a_2$ must be distinct elements of $\{13, 37, 53, 2549\}$, the ordered tuple $(a_i)_{i=1}^2$ can be chosen in $12$ ways.  Applying a similar argument to count the possible $(b_i)_{i=1}^3$, $(c_i)_{i=1}^2$, $(d_i)_{i=1}^5$, and the $(e_1)$, we see that the ordered tuple $(b_i)_{i=1}^3$ can be chosen in $12$ ways, the $(c_i)_{i=1}^2$ in $12$ ways, the $(d_i)_{i=1}^5$ in $120$ ways, and $(e_1)$ in $3$ ways. Thus there are \numprint{1 244 160} distinct ways to produce a quasiprime matrix $\QP$, which proves (1).

\bigskip
\noindent \textbf{Proof of (2):} Unfortunately, we only proved this by computational exhaustion using \texttt{Mathematica}. See part~(a) of Question~\ref{question:every_hidden_forest_can_be_found_via_QP_matrix}.

\bigskip
\noindent \textbf{Proof of (3):} Consider an arbitrary $\QP$.  Suppose $y_0$ is a solution to the five linear congruences $y+k \equiv 0 \pmod{C_k}$ where $C_k$ equals the product of the column entries of $\QP$ for $1 \leq k \leq 5$.  Setting $y_k = y_0 +k$, we observe that $y_5 \equiv 0 \pmod{6 \cdot d_5}$, and thus $y_5 \equiv 0 \pmod{2}$ and $y_5 \equiv 0 \pmod{3}$.  Hence $y_5 \in 2\mathbb{Z} \cap 3\mathbb{Z}$.  Since $y_5 \in 2\mathbb{Z}$, it follows that $y_3 = y_5-2$ implies $y_3 \in 2\mathbb{Z}$, and $y_1 = y_5-4$ implies $y_1 \in 2\mathbb{Z}$.  Moreover since $y_5 \in 3\mathbb{Z}$, it follows that $y_2 = y_5-3$ implies $y_2 \in 3\mathbb{Z}$. Thus (3) holds.

\bigskip

We wrote a program in \texttt{Mathematica} which applies the \texttt{CRT}-algorithm to each of the possible \numprint{1 244 160} matrices.  Four minutes later, the program yields that the smallest $y$-value solution is given by the following quasiprime matrix:
\begin{equation}\label{matrix:5x5_QP_Matrix}
\QP = \left(
\begin{array}{ccccc}
1 & 37 & 1 & 13 & 1 \\
31 & 1 & 5 & 269 & 1 \\
1 & 109 & 1 & 7 & 1 \\
67 & 17 & 41 & 23 & 11 \\
1 & 1 & 1 & 89 & 6 \end{array}
\right).
\end{equation}
This $y$-value and the next four consecutive integers have the following prime factorizations (with commas omitted in the factorizations for readability):
\begin{align*}
y_1 = \numprint{2 546 641 254 872 348} &= 2^2 \cdot 31 \cdot 67 \cdot 461 \cdot 664921471\\
y_2 = \numprint{2 546 641 254 872 349} &= 3^2 \cdot 17 \cdot 37 \cdot 109 \cdot 8681 \cdot 475421\\
y_3 = \numprint{2 546 641 254 872 350} &= 2 \cdot 5^2 \cdot 41 \cdot 11113 \cdot 111784759\\
y_4 = \numprint{2 546 641 254 872 351} &= 7^2 \cdot 13 \cdot 23 \cdot 73 \cdot 89 \cdot 269 \cdot 271 \cdot 367\\
y_5 = \numprint{2 546 641 254 872 352} &= 2^5 \cdot 3 \cdot 11 \cdot 2411592097417.
\end{align*}
Comparing the $x_1, \ldots, x_5$ with the $y_1, \ldots, y_5$ we see that $\gcd(x_i,y_j)>1$ for all $1 \leq i,j \leq 5$ and in fact $H^5_{(x_1,y_1)}$ has the following $\gcd$-matrix $\mathrm{Gcd}_M$ and corresponding matrix $M$:
$$\mathrm{Gcd}_M = \left(
\begin{array}{ccccc}
2 & 3 & 2^2 & 11 & 2 \! \cdot \! 3 \\
13 & 269 & 7 & 23 & 89 \\
2 & 5 & 2 & 41 & 2 \\
37 & 3 & 109 & 17 & 3 \\
2 & 31 & 2^2 & 67 & 2 \end{array}
\right)
\; \xmapsto[90^\circ \; \mathrm{right}]{\RotateRight} \;
M = \left(
\begin{array}{ccccc}
2 & 37 & 2 & 13 & 2 \\
31 & 3 & 5 & 269 & 3 \\
2^2 & 109 & 2 & 7 & 2^2 \\
67 & 17 & 41 & 23 & 11 \\
2 & 3 & 2 & 89 & 2 \! \cdot \! 3 \end{array}
\right).
$$

\begin{remark}
If we apply the \texttt{QP}-algorithm to the $M$ above, then we are forced to place a $2^2$ in either the $(3,1)$- or $(3,5)$-entry of $\QP$, and consequently the $6$ in the $(5,5)$ entry becomes a 3.  Hence this new $\QP$ differs from the quasiprime matrix in~(\ref{matrix:5x5_QP_Matrix}).  However, applying the \texttt{CRT}-algorithm to this new $\QP$ gives the same hidden forest as expected.
\end{remark}

\begin{remark}
The forest $\Hforest{5}$ with $x=\numprint{129 963 314}$ and $y=\numprint{2 546 641 254 872 348}$ is at a distance $d \approx 2.54664 \times 10^{15}$ from the origin.  Had we used the only known method until now (that is, Theorem~\ref{thm:2by2_case}), then we get a forest $\Hforest{5}$ with the following $x$ and $y$ values:
\begin{align*}
x &= \numprint{251 080 644 933 696 940 130 615 676 720 763 950}\\
y &= \numprint{108 580 359 501 475 197 963 484 708 875 960 338}.
\end{align*}
This forest is at a distance $d \approx 2.73553 \times 10^{35}$ from the origin, and hence is $1.07417 \times 10^{20}$ times farther than the forest we reveal in this paper!  We have not found a computationally tractable method to find the closest $5 \times 5$ hidden forest, nor do we believe that anyone else has.  So for the time being, the $\Hforest{5}$ we present in this paper is the closest $5 \times 5$ hidden forest to date.
\end{remark}


\section{Open problems and progress on recent research}\label{sec:open_problems}

There are many avenues for further research motivated from the work in this present paper. In this section, we give not only some open problems identified during our research process, but also some recent progress in generalizations of lattice point visibility.

\subsection{Open problems}

\begin{question}\label{question:every_hidden_forest_can_be_found_via_QP_matrix}
Is it true that for every hidden forest $\Hforest{n}$, there exists a quasiprime matrix $\QP$ in $\Mat_n(\mathbb{Z})$ such that the \texttt{CRT}-algorithm applied to $\QP$ yields $\Hforest{n}$? Related to this question are the following subquestions:
\begin{enumerate}[(a)]
\item Why do all \numprint{1 244 160} distinct quasiprime matrices in Matrix~(\ref{matrix:5x5}) yield exactly the same $x$-value solution under the \texttt{CRT}-algorithm?
\item Do all distinct quasiprime matrices produce unique solutions?
\item Can one code a computationally efficient method to search for the closest $\Hforest{n}$ for $n \geq 4$?
\end{enumerate}
\end{question}

\begin{question}
Higher dimensional analogues of patches of invisible points can be found.  Observe that our proof of Proposition~\ref{prop:fraction_of_visible_points} can easily be extended to higher dimensions by setting the value $s$ (in the proof) to the appropriate dimension.  That is, the probability that $(x_1, x_2, \ldots, x_s)$ is visible in $\mathbb{Z}^s$ is $\frac{1}{\zeta(s)}$.  In Example~\ref{exam:higher_dimension}, we find a hidden $2 \times 2 \times 2$ forest using a 3-dimensional analogue of the \texttt{CRT}-algorithm, and we see that the forest found by this method is very far from the origin.  Can we generalize the quasiprime matrix to these higher dimensional settings and find closer hidden $n$-dimensional forests?
\end{question}

\begin{example}\label{exam:higher_dimension}
In Figure~\ref{fig:three_dimensional_example}, we give an example of a hidden $2 \times 2 \times 2$ forest with corner point $(x_1,y_1,z_1)$ at $x_1=\numprint{9 126 194}$, $y_1=\numprint{8 286 564}$, and $z_1=\numprint{8 822 099}$.

\tikzset{node distance=3cm, auto}
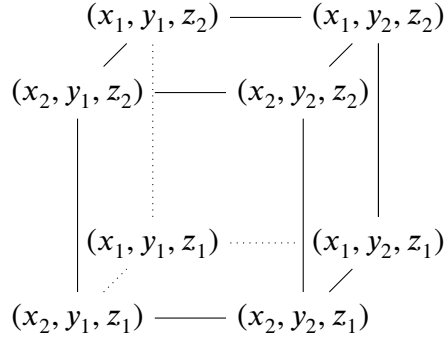
\begin{figure}[H]
\centering
\begin{tikzpicture}[%
  back line/.style={dotted},
  cross line/.style={preaction={draw=white, -,line width=6pt}}]
  \node (A) {$(x_2, y_1, z_2)$};
  \node [right of=A] (B) {$(x_2, y_2, z_2)$};
  \node [below of=A] (C) {$(x_2, y_1, z_1)$};
  \node [right of=C] (D) {$(x_2, y_2, z_1)$};
 
  \node (A1) [right of=A, above of=A, node distance=1cm] {$(x_1, y_1, z_2)$};
  \node [right of=A1] (B1) {$(x_1, y_2, z_2)$};
  \node [below of=A1] (C1) {$(x_1, y_1, z_1)$};
  \node [right of=C1] (D1) {$(x_1, y_2, z_1)$};
 
  \draw[back line] (D1) -- (C1) -- (A1);
  \draw[back line] (C) -- (C1);
  \draw[cross line] (D1) -- (B1) -- (A1) -- (A)  -- (B) -- (D) -- (C) -- (A);
  \draw (D) -- (D1) -- (B1) -- (B);
\end{tikzpicture}
\caption{A $2 \times 2 \times 2$ hidden forest}\label{fig:three_dimensional_example}
\end{figure}

To find this 3-dimensional hidden forest, we considered a 3-dimensional version of the prime matrix as a cube whose corners contain the first 8 prime numbers.  Then to each face of the cube, we multiplied the four numbers in each corner as the following image illustrates.
\begin{center}
\includegraphics[width=5in]{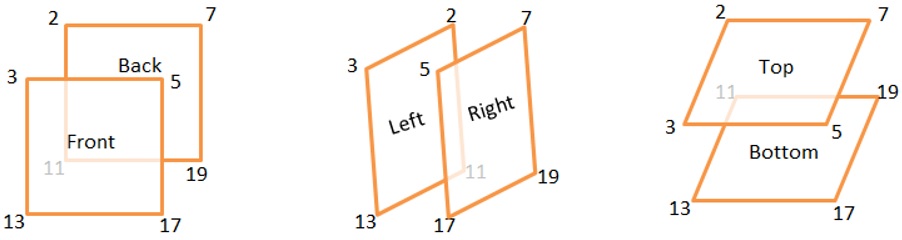}
\end{center}
Solving the following three pairs of systems of congruences
$$
\begin{cases}
x+1 \equiv 0 \pmod{\mbox{Back}}\\
x+2 \equiv 0 \pmod{\mbox{Front}}
\end{cases}
\hspace{.1in}
\begin{cases}
y+1 \equiv 0 \pmod{\mbox{Left}}\\
y+2 \equiv 0 \pmod{\mbox{Right}}
\end{cases}
\hspace{.1in}
\begin{cases}
z+1 \equiv 0 \pmod{\mbox{Bottom}}\\
z+2 \equiv 0 \pmod{\mbox{Top}}
\end{cases}
$$
yields the three simultaneous solutions $x_0=\numprint{9 126 193}$, $y_0=\numprint{8 286 563}$, and $z_0=\numprint{8 822 098}$. Then the following values $x_1$, $y_1$, $z_1$, $x_2$, $y_2$, and $z_2$ have the prime factorizations
\begin{align*}
x_1 &= 2 \cdot 7 \cdot 11 \cdot 19 \cdot 3119 &\hspace{.1in} y_1 &= 2^2 \cdot 3 \cdot 11^2 \cdot 13 \cdot 439 &\hspace{.1in} z_1 &= 11 \cdot 13 \cdot 17 \cdot 19 \cdot 191\\
x_2 &= 3 \cdot 5 \cdot 13 \cdot 17 \cdot 2753 & \hspace{.1in} y_2 &= 5 \cdot 7 \cdot 17 \cdot 19 \cdot 733 & \hspace{.1in} z_2 &= 2^2 \cdot 3 \cdot 5^2 \cdot 7 \cdot 4201.
\end{align*}
It is readily verified from these factorizations that each of the eight tuples of coordinates $(x_i, y_j, z_k)$ for $1 \leq i, j, k \leq 2$ have the property $\gcd(x_i, y_j, z_k) > 1$. Hence this 3-dimensional forest is indeed hidden from the origin.

\end{example}

\begin{question}
What can be said about hidden forests in the $\mathbb{Z}[i] \times \mathbb{Z}[i]$ lattice?  What is meant by the coordinate values $(x,y) \in \mathbb{Z}[i] \times \mathbb{Z}[i]$ being relatively prime?  Recall that if $R$ is a Euclidean domain (as is the case for the ring $\mathbb{Z}[i]$ of Gaussian integers), then greatest common divisors can be computed using the Euclidean algorithm.  Can we apply methods in this paper to the visibility of points in the lattice $\mathbb{Z}[i] \times \mathbb{Z}[i]$?
\end{question}

\subsection{Progress on recent research}
The following open problem was initially started by the second author Mbirika and his colleagues Pamela Harris and Bethany Kubik during their Visiting Assistant Professor appointments at West Point Military Academy in the summer of 2015. This question below has been recently explored in 2018 by Goins, Harris, Kubik, and Mbirika in~\cite{GHKM2017}.

\begin{question}[Harris, Kubik, Mbirika]
The classic setting focuses on integer lattice points which lie on straight lines through the origin with rational slopes.  We generalize this notion of lines of sights to include all curves through the origin given by power functions of the form $f(x) = ax^b$ where $a \in \mathbb{Q}$ and $b \in \mathbb{N}$.  What can we conclude about lattice point visibility in this generalized setting?  To begin to answer this question, we establish the following criterion for $b$-(in)visibility.
\end{question}

\begin{definition}[Visible and invisible lattice points] Fix $b\in \mathbb{N}$. A point $(r,s)\in \mathbb{N}\times\mathbb{N}$ is said to be \textit{$b$-invisible} if the following two conditions hold:
\begin{enumerate}[(1)]
\item The point $(r,s)$ lies on the graph of $f(x)=ax^b$ for some $a\in \mathbb{Q}$.  That is $s=ar^b$.
\item There exists an integer $k>1$ such that $k$ divides $r$ and $k^b$ divides $s$.
\end{enumerate}
The point is said to be \textit{$b$-visible} if it satisfies Condition (1) but fails to satisfy Condition (2).
\end{definition}
To speak about the $b$-visibility of a lattice point in this new setting, we develop a generalization of the greatest common divisor.
\begin{definition}\label{def:ggcd}
Fix $b \in \mathbb{N}$.   
The \textit{generalized greatest common divisor} of $r$ and $s$ with respect to $b$ is denoted $\ggcd_b$ and is defined as
$$\ggcd_b(r,s) := \max\{k \in \mathbb{N} \mid k \textrm{ divides } r \textrm{ and } k^b \textrm{ divides } s\}.$$
\end{definition}
The following result gives a necessary and sufficient condition to determine $b$-visibility.
\begin{proposition}\label{prop:visibility_criterion}
A point $(r,s)\in \mathbb{N}\times\mathbb{N}$ is $b$-visible if and only if $\ggcd_b(r,s)=1$.
\end{proposition}

Figure~\ref{fig:classic_and_generalized_curves} demonstrates both the classic and generalized setting. The \textcolor{red}{\bf red} curve \textcolor{red}{\bf $f(x)=7x$} represents the classic setting while the \textcolor{blue}{\bf blue} and \textcolor{green}{\bf green} curves \textcolor{blue}{\bf $g(x)=x^2$} and \textcolor{green}{\bf $h(x)=\frac{1}{7}x^3$}, respectively, represent the generalized setting.


\begin{figure}[h!]
\begin{center}
\begin{tikzpicture}[yscale=.1,xscale=1.75, domain=0:7.883735]  

\foreach \x in {1,...,8}
		\draw [very thin] (\x,0) -- (\x,70);

\foreach \x in {7,14,...,63,70}
		\draw [very thin] (0,\x) -- (8,\x);
		
\draw [<->, ultra thick, brown] (0,75) -- (0,0) -- (8.25,0);

\draw[very thick,red] (0,0) -- (8,56);
\draw[very thick,blue] (0,0) parabola (8,64);
\draw[very thick,green] plot (\x,{(1/7)*\x^3});

\node at (1,7) {\tikzcircle[fill=black]{2.5pt}};
\foreach \x in {2,...,6}
	\node at (\x,\x*7) {\tikzcircle[fill=white]{2.5pt}};
\node at (8,56) {\tikzcircle[fill=white]{2.5pt}};

\node at (1,1) {\tikzcircle[fill=black]{2.5pt}};
\foreach \x in {2,...,6}
	\node at (\x,\x*\x) {\tikzcircle[fill=white]{2.5pt}};
\node at (8,64) {\tikzcircle[fill=white]{2.5pt}};

\node at (7,49) {\tikzcircle[fill=white,white]{2.75pt}}; 
\node at (7,49) {$\circlerighthalfblack$};	

\node [above left] at (1,6.5) {\small(1,7)};
\node [above left] at (2,13.5) {\small(2,14)};
\node [above left] at (3,20.5) {\small(3,21)};
\node [above left] at (4,27.5) {\small(4,28)};
\node [above left] at (5,34.5) {\small(5,35)};
\node [above left] at (6,41.5) {\small(6,42)};

\node [above left] at (8,55.5) {\small(8,56)};

\node [above left] at (1,.5) {\small(1,1)};
\node [above left] at (2,3.5) {\small(2,4)};
\node [above left] at (3,8.5) {\small(3,9)};
\node [above left] at (4,15.5) {\small(4,16)};
\node [above left] at (5,24.5) {\small(5,25)};
\node [above left] at (6,35.5) {\small(6,36)};
\node [above left] at (7,48.5) {\small(7,49)};
\node [above left] at (8,63.5) {\small(8,64)};

\foreach \x in {1,...,8}
		\node [below] at (\x,-.75) {\x};
\foreach \x in {7,14,...,63,70}
		\node [left] at (-.075,\x) {\x};
\node at (0,0) {\includegraphics[width=8ex]{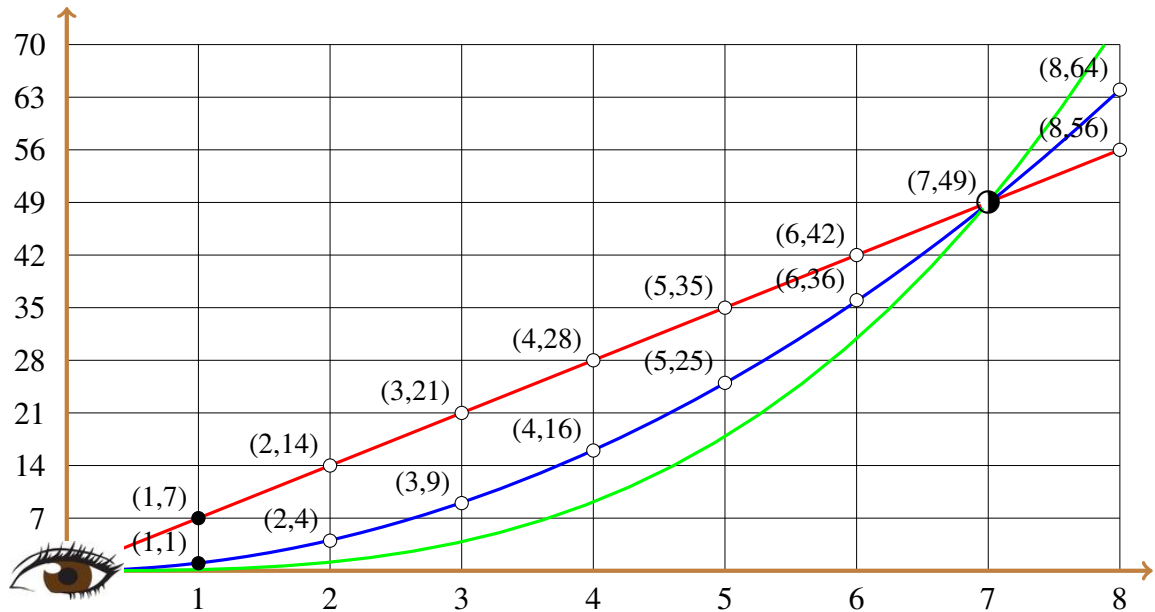}};

\end{tikzpicture}
\end{center}
\caption{Three lines of sight $f(x) = 7x$, $g(x) = x^2$, and $h(x) = \frac{1}{7}x^3$.}
\label{fig:classic_and_generalized_curves}
\end{figure}

Observe that the point $(7,49)$ is not 1-visible since $\gcd(7,49)=7$ and is not 2-visible since $\ggcd_2(7,49)=7$. However the point $(7,49)$ is 3-visible since $\ggcd_3(7,49)=1$.

\begin{theorem}[Goins, Harris, Kubik, Mbirika~\cite{GHKM2017}]
Fix an integer $b\in\mathbb{N}$.  Then the proportion of points $(r,s)\in\mathbb{N}\times\mathbb{N}$ that are $b$-visible is $\displaystyle\frac{1}{\zeta(b+1)}$.
\end{theorem}

\begin{theorem}[Goins, Harris, Kubik, Mbirika~\cite{GHKM2017}]
For every $m,n,b\in\mathbb{N}$, there exists $b$-invisible $n\times m$ forests.
\end{theorem}

\begin{question}
Can we apply the new techniques detailed in this paper to find the closest $b$-invisible $n \times n$ forests?
\end{question}




\newpage


\begin{appendices}
\renewcommand{\thesection}{A}
\section*{Appendix}
\tocless
\subsection{Java code to verify closest hidden forests\label{appendix_A}}
In this appendix we provide the Java code that we wrote to exhaustively search the integer lattice for the closest hidden forests.

  \begin{lstlisting}[frame=single,language=Java,caption={Java code to search the integer lattice for hidden forests}]
package project1;
 
import java.util.Scanner;
 
public class Compiler {
 
      public static void main(String[] args) {
            Scanner in = new Scanner(System.in);
            // long is a number;
            // get bottom, leftmost, and rightmost from user input;
						
            System.out.println("What is the bottom?");
            long bottom = in.nextLong();
						
            System.out.println("What is the leftmost?");
            long leftmost = in.nextLong();
						
            System.out.println("What is the rightmost?");
            long rightmost = in.nextLong();
						
            long boxWidth = rightmost - leftmost;
            long count = 0;
            long foundCount = 0;
						
            System.out.println("How often do you want to check?");
            long check = in.nextLong();
						
            System.out.println("Starting");
            boolean exit = false;
						
            while (bottom < Long.MAX_VALUE - boxWidth && !exit) {
                  boolean equalsOne = false;
                  long startY = bottom+1;
                  for (long y = bottom; y < (startY + boxWidth) && !equalsOne
                              && !exit; ++y) {
                        ++count;
                        for (long x = leftmost; (x <= rightmost) && !equalsOne && !exit; ++x) {
                              ++foundCount;
                              if (test(x, y) == 1) {
                                    bottom = y;
                                    equalsOne = true;
                                    foundCount = 0;
                                    x = leftmost;
                              }// test end
                              if (foundCount == ((boxWidth+1)*(boxWidth+1))) {
                                    System.out.println("found one!! upper right corner = x:"
                                                + x + " y:" + y);
                                    exit = true;
                                    foundCount = 0;
                              }// foundCount end
                        }// for x end
                        if (count % check == 0) {
                              System.out.println(y + " is current y");
                        }// count%check end
                  }// for y end
                  ++bottom;
            }// while 1 end
      }// main end
 
      private static long test(long x, long y) {
            if (y == 0)
                  return x;
            return test(y, x % y);
      }
}// compiler end
  \end{lstlisting}
	
\end{appendices}

\end{document}